\documentclass[11pt]{article}
\usepackage{}
\usepackage{enumerate}
\usepackage{mathbbold}
\usepackage{bbm}
\usepackage{amsfonts}
\usepackage[T1]{fontenc}
\usepackage{latexsym,amssymb,amsmath,amsfonts,amsthm}
\usepackage{graphicx, color}
\usepackage{subfigure}
\usepackage{epsf}
\usepackage{epstopdf}
\usepackage{amssymb}
\usepackage{diagbox}
\usepackage{mathrsfs}
\usepackage[ruled]{algorithm2e}
\renewcommand{\theequation}{\arabic{section}.\arabic{equation}}

\newcommand{\R}{{\mathbb R}}

\newcommand{\be}{\begin{eqnarray}}
\newcommand{\ben}{\begin{eqnarray*}}
\newcommand{\en}{\end{eqnarray}}
\newcommand{\enn}{\end{eqnarray*}}

\newcommand{\pa}{\partial}

\newcommand{\curl}{{\rm curl\,}}

\newcommand{\ggrad}{{\rm grad\,}}
\newcommand{\ddiv}{{\rm div\,}}

\newcommand{\Om}{\Omega}
\newcommand{\om}{\omega}

\newcommand{\hx}{\hat{x}}

\newtheorem{theorem}{Theorem}[section]

\begin{document}
\title{\bf A quantitative sampling method for elastic and electromagnetic sources}
\author{
Xiaodong Liu\thanks{State Key Laboratory of Mathematical Sciences, Academy of Mathematics and Systems Science,
Chinese Academy of Sciences, Beijing 100190, China. Email: xdliu@amt.ac.cn}
,\and
Qingxiang Shi\thanks{Corresponding author. Yau Mathematical Sciences Center, Tsinghua University,
Beijing 100084, China. Email: sqxsqx142857@tsinghua.edu.cn}
}
\date{}
\maketitle

\begin{abstract}
This work is dedicated to a novel sampling method for accurately reconstructing elastic and electromagnetic sources from the far field patterns. We show that the proposed indicators in the form of integrals with full far field patterns are exactly the source functions. These facts not only give constructive uniqueness proofs of the inverse source problems, but also establish the theoretical basis of the proposed sampling methods. Furthermore, we derive the stability estimates for the corresponding discrete indicators using the far field patterns with finitely many observations and frequencies.
We have also proposed the indicators with partial far field patterns and proved their validity for providing the derivative information of the unknown sources. Numerical examples are presented to verify the accuracy and stability of the proposed quantitative sampling method. 

\vspace{.2in}

\vspace{.2in} {\bf AMS subject classifications:}
35P25, 78A46, 74B05
\end{abstract}

\section{Introduction}
Inverse source problems (ISPs) are concerned with the reconstruction of an unknown source from some kind of measurement induced by source. The ISPs arise in many important scientific and industrial applications including medical tomography \cite{apply-at-2007,apply-at-2009}, antenna technology \cite{apply-at-1991,apply-at-1982} and seismic monitoring \cite{apply-at-1986}. 
In the last decades, many works have been done in dealing with the multi-frequency measurements about uniqueness and stability \cite{full-thm-2015,full-thm-2016,full-thm-2017,full-thm-2020}  as well as numerical methods, for example, the Recursive algorithm \cite{full-num-recu}, the Fourier method \cite{full-num-fouri-Maxwell-1,full-num-fouri,full-num-fouri-Maxwell-2,full-num-fouri-Elastic}, the Variational Bayes's approach \cite{full-num-learn} and the sampling type methods \cite{spar-2020, full-num-facto, JiLiu-sisc2019,JiLiu-siap2021,spar-2023LiLiu,spar-2023,LiuMeng-csiam2023}. 

Among various techniques for object reconstruction, the sampling type methods have many advantages, for example, low computational complexity, easy implementation, strong robustness and less dependence on the prior physical/geometric information. However, most of the sampling type methods are qualitative methods for inverse scattering problems. The objective is often not to give a complete description (e.g, the precise value of the source function or the refractive index of an inhomogeneity) of the scatterer, but rather to determine certain features (e.g. an approximation of the location and shape) of the scatterer. We refer to \cite{spar-2020, full-num-facto, JiLiu-siap2021, spar-2023LiLiu, spar-2023}, where the source support is roughly reconstructed. 
A highlight of these direct sampling methods is that the measurements are only taken at finitely many sensors, which is attractive in many practical cases. Note that the rough reconstructions provide good initial guess for optimization techniques. Therefore, for many practical applications the combination of qualitative methods and optimization methods is expected to produce high precision reconstructions.

The main contribution of this paper is to design a quantitative sampling method for accurately reconstructing the  source functions. 
For point sources, some indicators are designed in the recent works \cite{JiLiu-sisc2020,spar-2023LiLiu,spar-2023,spar-point-2023} to determine all the information, i.e., the number, their locations and the scattering strengths. In other words, the corresponding sampling methods are quantitative methods that can be used to identify the point sources. However, an important a priori information is that the point sources are Dirac measures, which is actually a key for the successful numerical simulations. Therefore, one can not use such quantitative methods for a general source. We thus have to design novel quantitative sampling method for identifying a general source function. We want to remind the other quantitative method named the Fourier method \cite{full-num-fouri-Maxwell-1,full-num-fouri,full-num-fouri-Maxwell-2,full-num-fouri-Elastic} which is based on an observation that the coefficients of the Fourier expansion of the source function are the far field patterns with some specific observation directions and frequencies. However, instead of using  equispaced data with respect to the products of observations and frequencies, our sampling method makes use of the far field patterns with simply chosen observations and frequencies. 

We take the elastic and electromagnetic sources for our modalities since both of them involves scattered fields in the form of vector fields. The corresponding uniqueness and stability estimates for the inverse sources scattering problems have recently been established in \cite{stability-elastic-electromagnetic, non-sou-EM}. 
Our novel quantitative sampling method makes use of the full or partial far field patterns of the scattered fields. 
The theoretical basis of the novel sampling method for stably computing the source functions will be established. Precisely, for elastic source problem, we design an indicator function and prove that the source can be reconstructed accurately by the indicator function with full (compressional and shear) far field patterns. When only compressional or shear far field patterns are available, we also design the indicators to reconstruct some first-order derivatives information of the source function. 
Furthermore, for electromagnetic source problem with given free charge density, we design an indicator function with electric far field patterns to reconstruct the source accurately. If the free charge density is unknown, another indicator with magnetic far field patterns is designed to reconstruct the curl of the source function. 
Theoretically, to identify the extended source function, the far field patterns have to be taken at all the observation directions. However, from the numerical point of view, one has to consider the indicators in the form a finite sum with finitely many observation directions and frequencies. To verify the effectiveness of our method, we derive the corresponding stability estimates for these discrete indicators. 
Numerous numerical examples show that our method stably provides high resolution reconstructions even if there are only dozens of directions and frequencies.

This paper is organized as follows. In section $2$, we recall the direct elastic source scattering problem, introduce the indicators for identifying the elastic sources and prove their theoretical basis. The analogous results for electromagnetic sources are studied in Section $3$. The stability estimates for the discrete indicators are established in  Section $4$. Numerical experiments are then presented in the final section, showing the validity and robustness of the quantitative sampling method.

\section{A quantitative sampling method for elastic sources}

\subsection{Direct elastic source scattering problems}
We begin with some notations in $\R^2$. For $x=(x_1,x_2)^T\in\mathbb R^2$, we define $\hx^{\perp}$ obtained by rotating $\hat{x}:=x/|x|$ anticlockwise by $\pi/2$ as well as two auxiliary differential operators $\ggrad^{\perp}$ and $\ddiv^{\perp}$ by
$\ggrad^{\perp}:= (-\pa_2, \pa_1)^T$ and $\ddiv^{\perp}:= (-\pa_2, \pa_1)$, respectively.

Without loss of generality, we always assume that the density of the elastic medium is $\rho=1$. For a source $S\in \left(L^2\left(\mathbb R^n\right)\right)^n$ with compact support $\Omega$ in an isotropic homogeneous medium, the propagation of time-harmonic  wave $u$ solves the Navier equation
\begin{align}
    \label{ela}
    \mu\Delta u +(\lambda+\mu){\rm grad \,div} u+ \omega^2 u=-S\quad{\rm in} \ \mathbb R^n,
\end{align}
where $\omega$ is the circular frequency, and Lam\'e constants $\lambda$, $\mu$ satisfy
$\mu>0,\ 2\mu+\lambda>0$. Moreover, we denote by
\ben
k_p:=\frac{\omega}{\sqrt{\lambda+2\mu}}\quad\mbox{and}\quad k_s:=\frac{\omega}{\sqrt{\mu}},
\enn
respectively, the compressional wave number and the shear wave number. 
Then the solution $u$ of equation \eqref{ela} has a decomposition in the form
 \begin{align*}
     u=u_p+u_s,
 \end{align*}
where
\begin{equation*}
          u_p:=-\frac{1}{k_p^2}{\rm grad\ div}u\quad\mbox{and}\quad 
          u_s:=\left\{
          \begin{aligned}
          -\frac{1}{k_s^2}{\rm grad}^\perp\ {\rm div}^\perp u  \quad&{\rm in}\ \mathbb R^2,\\
          \frac{1}{k_s^2} {\rm curl\ curl}\ u\quad&{\rm in} \ \mathbb R^3,
          \end{aligned}
          \right.
\end{equation*}
are known as the compressional and shear parts of $u$, respectively.
The scattered fields $u_p$ and $u_s$ satisfy the Kupradze's radiation conditions
\be\label{KupradzeRC}
\frac{\partial u_t}{\partial |x|}-ik_t u_t=o(|x|^{-\frac{n-1}{2}}),\quad |x|\rightarrow\infty, t=p,s,
\en
uniformly in all directions $\hx:=x/|x|\in \mathbb{S}^{n-1}:=\{y\in\R^n\,|\, |y|=1\}$. 
Denoting by $e_j, 1\leq j\leq n$, the unit cartesian vectors in $\mathbb R^n$ we define for $x\in\R^n, x\neq y$, the $jth$ column of $\Gamma(x,y)$ by 
\ben
    \Gamma(x,y)e_j=\left\{
    \begin{aligned}
    &\frac{i H^1_0(k_s|x-y|)}{4\mu}e_j+\frac{i}{4\omega^2}{\rm grad}_x{\rm div}^\perp_x\left\{\left[H^1_0(k_s|x-y|)-H^1_0(k_p|x-y|) \right]e_j\right\}& {\rm in}\  \mathbb R^2,\\
    &\frac{e^{ik_s|x-y|}}{4\pi\mu|x-y|}e_j+\frac{1}{\omega^2}{\rm grad}_x{\rm div}_x\left[ \frac{e^{ik_s|x-y|}-e^{ik_p|x-y|}}{4\pi|x-y|}e_j\right]& {\rm in} \ \mathbb R^3.
    \end{aligned}
    \right.
\enn
where $H^1_0$ denotes the Hankel function of the first kind of order 0. 
Now, we have the representation for the solution of the problem \eqref{ela}-\eqref{KupradzeRC} in the form
\begin{align}
    u(x)=\int_\Omega \Gamma(x,y)S(y)dy,\quad x\in\R^n.
    \label{ source solution}
\end{align}

The inverse problem is to determine the source from a knowledge of the far field patterns. To do so, we collect the representations for the far field patterns.
In $\R^2$, the scattered field $u$ has the following asymptotic behavior
\begin{align*}
    u(x,\om)=\frac{1}{\lambda+2\mu}\frac{e^{i\pi/4}}{\sqrt{8\pi k_p}}\frac{e^{ik_p|x|}}{\sqrt{|x|}}u^\infty_p(\hat{x},\om)+\frac{1}{\mu}\frac{e^{i\pi/4}}{\sqrt{8\pi k_s}}\frac{e^{ik_s|x|}}{\sqrt{|x|}}u^\infty_s(\hat{x},\om)+O\left(\frac{1}{|x|^{3/2}}\right),\quad |x|\to\infty,
\end{align*}
where 
\begin{equation}
\left\{
     \begin{aligned}
      \label{exsin2D}
        u^\infty_p(\hat{x},\om)=&\hat{x}\int_\Omega e^{-ik_p\hat{x}\cdot y}S(y)\cdot\hat{x}dy,\quad \hx\in \mathbb{S}^1,\,\om\in(0,+\infty), \\
    u^\infty_s(\hat{x},\om)=&\hx^\perp\int_\Omega e^{-ik_s\hat{x}\cdot y} S(y)\cdot \hx^\perp dy,\quad \hx\in \mathbb{S}^1,\,\om\in(0,+\infty).
    \end{aligned}
    \right.
\end{equation}
In $\R^3$, the scattered field $u$ has the following asymptotic behavior
\begin{align*}
    u(x,\om)=\frac{1}{4\pi(\lambda+2\mu)}\frac{e^{ik_p|x|}}{|x|}u^\infty_p(\hat{x},\om)+\frac{1}{4\pi\mu}\frac{e^{ik_s|x|}}{|x|}u^\infty_s(\hat{x},\om)
            +O\left(\frac{1}{|x|^{2}}\right),\quad |x|\to\infty,
\end{align*}
where
\begin{equation}
\left\{
    \begin{aligned}
\label{exsin3D}
     u^\infty_p(\hat{x},\om)=&\hat{x}\int_\Omega e^{-ik_p\hat{x}\cdot y} S(y)\cdot\hat{x} dy,\quad \hx\in \mathbb{S}^2,\,\om\in(0,+\infty),\\
    u^\infty_s(\hat{x},\om)=&\int_\Omega e^{-ik_s\hat{x}\cdot y}  \hat{x}\times(S(y)\times\hat{x})dy,\quad \hx\in \mathbb{S}^2,\,\om\in(0,+\infty).
\end{aligned}
\right.
\end{equation}
Note that $u^\infty_p$ and $u^\infty_s$ are known as, respectively, the compressional and shear far field pattern of $u$. 
Both $u_p^\infty$ and $u_s^\infty$ are defined in the form of a vector function and are naturally decoupled. 
Denote by  $u^\infty(\hx,\om):=[u^\infty_p(\hat{x},\om),u^\infty_s(\hat{x},\om)]$, $\hx\in\mathbb S^{n-1},\,\om\in(0,+\infty)$, the full far field patterns.

\subsection{Indicators with full or partial far field patterns for elastic sources}
Given the far field patterns $u^\infty(\hx,\om):=[u^\infty_p(\hat{x},\om),u^\infty_s(\hat{x},\om)]$, $\hx\in\mathbb S^{n-1},\,\om\in(0,+\infty)$, we introduce  three indicator functions $\mathcal{I}_f$, $\mathcal{I}_p$ and $\mathcal{I}_s$ by
\be
    \label{indicator-f}
    &&\mathcal{I}_f(z):=
    \frac{1}{(2\pi)^n}\int_{\mathbb S^{n-1}}\int_0^{+\infty}\left[u_p^{\infty}(\hx,\omega)+u_s^{\infty}\left(\hx,\frac{k_p}{k_s}\omega\right)\right]e^{ik_p\hx\cdot z}\frac{\omega^{n-1}}{(\lambda+2\mu)^{n/2}}d\omega ds_{\hx},\quad z\in\mathbb R^n,\qquad
    \\
    \label{indicator-p}
         &&\mathcal{I}_p(z):=
         \frac{i}{(2\pi)^n}\int_{\mathbb S^{n-1}}\int_0^{+\infty}\left(u_p^{\infty}(\hx,\omega)\cdot\hx\right)e^{ik_p\hx\cdot z} \frac{k_p^n}{\sqrt{\lambda+2\mu}}d\omega ds_{\hx},\quad z\in\mathbb R^n,
\en
and
\be
     \label{indicator-s}
        \mathcal{I}_s(z):=
        \left\{
        \begin{aligned}
            &\frac{i}{(2\pi)^2}\int_{\mathbb S^{1}}\int_0^{+\infty}\left(u_s^{\infty}(\hx,\omega)\cdot \hx^{\perp}\right)e^{ik_s\hx\cdot z}\frac{k_s^{2}}{\sqrt{\mu}}d\omega ds_{\hx},\quad z\in\mathbb R^2,\\
            &\frac{i}{(2\pi)^3}\int_{\mathbb S^{2}}\int_0^{+\infty}\left(u_s^{\infty}(\hx,\omega)\times\hx\right)e^{ik_s\hx\cdot z}\frac{k_s^{3}}{\sqrt{\mu}}d\omega ds_{\hx},\quad z\in\mathbb R^3,
        \end{aligned}   
        \right.
\en
respectively.

Here and throughout this paper, for any vector field $U\in\left(L^2(\mathbb R^n)\right)^n$, we use $\mathcal{F}$ and $\mathcal{F}^{-1}$ to denote its Fourier transform and inverse Fourier transform by
\ben
\mathcal{F}[U](y):=\int_{\R^n} U(x)e^{-ix\cdot y}dx
\quad\mbox{and}\quad
\mathcal{F}^{-1}[U](y):=\frac{1}{(2\pi)^n}\int_{\R^n} U(x)e^{ix\cdot y}dx, \quad y\in\R^n,
\enn
respectively.
\begin{theorem}\label{thm-f}
    Let $S\in\left(L^2(\mathbb R^n)\right)^n$, then
    \be\label{IfS}
    \mathcal{I}_f=S.
    \en
\end{theorem}
\begin{proof}
    Using \eqref{exsin2D} and \eqref{exsin3D}, we have
    \begin{align*}
        u_p^{\infty}(\hx,\omega)+u_s^{\infty}\left(\hx,\frac{k_p}{k_s}\omega\right)=\int_\Omega e^{-ik_p\hat{x}\cdot y} S(y) dy=\mathcal{F}[S](k_p\hx), \quad \hx\in\mathbb S^{n-1},\,\om\in(0,+\infty).
    \end{align*}
    Inserting this into \eqref{indicator-f}, the statement \eqref{IfS} follows by the straightforward calculations, i.e.,    
    \begin{align*}
        \mathcal{I}_f&=\frac{1}{(2\pi)^{n}}\int_{\mathbb S^{n-1}}\int_0^{+\infty}\mathcal{F}[S](k_p\hx)e^{ik_p\hx\cdot z}\frac{\omega^{n-1}}{(\lambda+2\mu)^{n/2}}d\omega ds_{\hx}\\
        &=\frac{1}{(2\pi)^{n}}\int_{\mathbb S^{n-1}}\int_0^{+\infty}\mathcal{F}[S](k_p\hx)e^{ik_p\hx\cdot z}k_p^{n-1}dk_p ds_{\hx}\\
         &=\frac{1}{(2\pi)^{n}}\int_{\mathbb R^{n}}\mathcal{F}[S](\xi)e^{i\xi\cdot z}d\xi\\
         &=\mathcal{F}^{-1}[\mathcal{F}[S]]\\
         &=S.
    \end{align*}
\end{proof}

Note that the proof of Theorem \ref{thm-f} is also a constructive unique proof of the inverse elastic source scattering problem from the full far field patterns $\{u^{\infty}(\hx,\omega)|\hx\in\mathbb S^{n-1},\omega\in(0,+\infty)\}$. Furthermore, Theorem \ref{thm-f} gives the theoretical basis for the indicator $\mathcal{I}_f$ to reconstruct the source $S\in\left(L^2(\mathbb R^n)\right)^n$ from the full far field patterns $\{u^{\infty}(\hx,\omega)|\hx\in\mathbb S^{n-1},\omega\in(0,+\infty)\}$. 
Practically, we may have only the compressional far field pattern $u^\infty_p$ or the shear far field pattern $u^\infty_s$.
The next theorem shows that we are still able to reconstruct partial information of source $S$ by $\mathcal{I}_p$ or $\mathcal{I}_s$. 

\begin{theorem}\label{thm-ps}
     Let $S\in\left(H^1(\mathbb R^n)\right)^n$, then
     \be\label{IpdivS}
     \mathcal{I}_p={\rm div}S
     \en
     and 
     \be\label{IsdivS}
     \mathcal{I}_s=\left\{
         \begin{aligned}
             {\rm div}^{\perp} S&,\ {\rm if}\ n=2,\\
             -{\rm curl}\  S &,\ {\rm if}\ n=3.
         \end{aligned}\right.
     \en
\end{theorem}
\begin{proof}
Using \eqref{exsin2D} and \eqref{exsin3D} again, we deduce that
    \begin{align*}
      \mathcal{I}_p(z)
         &=\frac{i}{(2\pi)^{n}}\int_{\mathbb S^{n-1}}\int_0^{+\infty}\left(\mathcal{F}[S](k_p\hx)\cdot\hx \right)e^{ik_p\hx\cdot z}\frac{k_p^{n}}{\sqrt{\lambda+2\mu}}d\omega ds_{\hx}\\
         &=\frac{i}{(2\pi)^{n}}\int_{\mathbb S^{n-1}}\int_0^{+\infty}\left(\mathcal{F}[S](k_p\hx)\cdot(k_p\hx)\right)e^{ik_p\hx\cdot z}k_p^{n-1}dk_p ds_{\hx}\\
         &=\frac{i}{(2\pi)^{n}}\int_{\mathbb R^{n}}\mathcal{F}[S](\xi)\cdot\xi e^{i\xi\cdot z}d\xi\\
         &=\frac{1}{(2\pi)^{n}}\int_{\mathbb R^{n}}\mathcal{F}[{\rm div}S](\xi) e^{i\xi\cdot z}d\xi\\
         &={\rm div}S(z),\quad z\in\R^n,
    \end{align*}
    \begin{align*}
      \mathcal{I}_s(z)
         &=\frac{i}{(2\pi)^2}\int_{\mathbb S^1}\int_0^{+\infty}\mathcal{F}[S](k_s\hx)\cdot x^{\perp} e^{ik_s\hx\cdot z}\frac{k_s^{2}}{\sqrt{\mu}}d\omega ds_{\hx}\\
        &=\frac{i}{(2\pi)^2}\int_{\mathbb S^{1}}\int_0^{+\infty}\mathcal{F}[S](k_s\hx)\cdot(k_sx^{\perp})e^{ik_p\hx\cdot z}k_sdk_s ds_{\hx}\\
         &=\frac{i}{(2\pi)^2}\int_{\mathbb R^{2}}\mathcal{F}[S](\xi)\cdot\xi^{\perp} e^{i\xi\cdot z}d\xi\\
         &=\frac{1}{(2\pi)^2}\int_{\mathbb R^{2}}\mathcal{F}[{\rm div}^{\perp}S](\xi) e^{i\xi\cdot z}d\xi\\
         &={\rm div}^{\perp} S(z),\quad z\in\R^2
    \end{align*}
and
  \begin{align*}
         \mathcal{I}_s(z)
         &=\frac{i}{(2\pi)^{3}}\int_{\mathbb R^{3}}\left(\mathcal{F}[S](\xi)\times\xi\right) e^{i\xi\cdot z}d\xi\\
         &=\frac{-1}{(2\pi)^{3}}\int_{\mathbb R^{3}}\mathcal{F}[{\rm curl}S](\xi)e^{i\xi\cdot z}d\xi\\
         &=- {\rm curl}\ S(z),\quad z\in\R^3,
    \end{align*}
where we have used
\begin{align*}
      i\mathcal{F}[S](\xi)\cdot\xi=\mathcal{F}[{\rm div}S](\xi),\quad  i\mathcal{F}[S](\xi)\cdot\xi^\perp=\mathcal{F}[{\rm div^\perp}S](\xi) \quad{\rm  and}\quad \mathcal{F}[S](\xi)\times\xi=i\mathcal{F}[{\rm curl}S](\xi)
\end{align*}
respectively.
\end{proof}
Before ending this section, we want to remark that both $\mathcal{I}_p$ and $\mathcal{I}_s$
vanish outside the support of the source $S$, which implies that if the source function $S$ is neither a gradient field nor a curl field, both $\mathcal{I}_p$ and $\mathcal{I}_s$ can be used to give an initial reconstruction of the source support $\Om$.


\section{A quantitative sampling method for electromagnetic  sources}
\subsection{Direct electromagnetic source scattering problems}
We consider an external source $J\in\left( L^2(\mathbb R^3)\right)^3$ with compact support $\Omega$.
For simplicity, we assume that the background medium is homogeneous and dielectric. In other words, electric conductivity $\sigma=0$, while both electric permittivity $\varepsilon$ and magnetic permeability $\mu$ are positive constants. Let $k:=\omega\sqrt{\mu\varepsilon}>0$ be wavenumnber with frequency $\om$. Then the external source $J$ gives the propagation of time-harmonic electromagnetic waves governed by the Maxwell equations
\be\label{ME}
\operatorname{curl} E-i \omega \mu H & =0\quad\mbox{and}\quad
\operatorname{curl} H+i \omega \varepsilon E & =J\quad {\rm in}\, \mathbb R^3,
\en
where $E$ and $H$ denote the  electric field and the magnetic field, respectively.
To characterize the outgoing waves, the scattered fields have to satisfy the Silver-Muller radiation condition
\begin{align}\label{SM-rc}
 \lim\limits _{|x| \rightarrow \infty}\left[\sqrt{\mu} H(x) \times x-|x| \sqrt{\varepsilon} E(x)\right]=0,
\end{align}
which holds uniformly w.r.t. $\hx:=x/|x|$.
Recall the dyadic Green's function 
\ben
    \mathbb{G}_k(x, y):=\Phi_k(x, y) \mathbb{I}+\frac{1}{k^2} \nabla_y^2 \Phi_k(x, y), \quad x \neq y,
\enn
for Maxwell's equations where $\mathbb{I}$ is a $3 \times 3$ identity matrix and $\nabla_y^2 \Phi_k(x, y)$ is the Hessian matrix for $\Phi_k$ with respect to $y$ and
$\Phi_k(x,y):=\frac{e^{ik|x-y|}}{4\pi|x-y|},\ x\neq y$.
Then, the unique solution $(E,H)$ of the problem \eqref{ME}-\eqref{SM-rc} has the representation
\ben
    E(x, k)=i\omega\mu \int_{\mathbb{R}^3} \mathbb{G}_k(x, y) J(y) d y, \quad 
    H(x,k)={\rm curl}\int_{\mathbb{R}^3} \mathbb{G}_k(x, y) J(y) d y,\quad x \in \mathbb{R}^3,\,k\in\R^{+} .
\enn
The scattered fields $E$ and $H$ have the following asymptotic form
\begin{align*}
 \begin{array}{ll}
E(x, k)=\frac{e^{i k|x|}}{|x|}\left\{E_{\infty}(\hat{x}, k)+O\left(\frac{1}{|x|}\right)\right\}, & |x| \rightarrow \infty, \\
H(x, k)=\frac{e^{i k|x|}}{|x|}\left\{H_{\infty}(\hat{x}, k)+O\left(\frac{1}{|x|}\right)\right\}, & |x| \rightarrow \infty,
\end{array}   
\end{align*}
uniformly in all directions $\hat{x}=x /|x|$ where the vector fields $E_{\infty}$ and $H_{\infty}$ defined on the unit sphere $\mathbb{S}^2$ are known as the electric far field pattern and magnetic far field pattern, respectively. 
Precisely, 
\begin{align}
\label{ME-exp}
    &E_{\infty}(\hx,k)=\frac{ik}{4\pi\sqrt{\varepsilon}} \hx\times\left(\int_{\R^3}e^{-ik\hx\cdot y}J(y)dy\times\hx\right),\quad \hx\in\mathbb S^2,\,k\in\R^{+},\\
    &H_{\infty}(\hx,k)= \sqrt{\frac{\varepsilon}{\mu}}\hx \times E_{\infty}(\hx),\quad \hx\in\mathbb S^2,\,k\in\R^{+}.    
\end{align}
Note that $E_{\infty}$ and $H_{\infty}$ are tangential vector fields and orthogonal to each other, so it is sufficient to work only with one of the far field patterns.

\subsection{Indicators for electromagnetic sources}
The electric far field patterns $E_{\infty}$ define a vector indicator function  \begin{align}
    \label{indicator-e}
   \mathcal{I}_E(z):=\frac{-i\sqrt{\varepsilon}}{2\pi^2}\int_{\mathbb S^2}\int_0^{+\infty}E_{\infty}(\hx,k)e^{ik\hx\cdot z}kdkds_{\hx},\ z\in\mathbb R^3.
\end{align}
The following theorem implies that, if the source function is divergence free, such an indicator function is nothing else but the external source function. 
\begin{theorem}
\label{thm-e}
    Let $J\in\left(L^2(\mathbb R^3)\right)^3$ be divergence free, i.e., ${\rm div}J=0$, we have
    \be
    \mathcal{I}_E=J.
    \en
\end{theorem}
\begin{proof}
Inserting \eqref{ME-exp} into \eqref{indicator-e}, we have
\begin{align*}
\mathcal{I}_E(z)
&=\frac{1}{(2\pi)^3}\int_{\mathbb S^2}\int_0^{+\infty}\hx\times\left(\int_{\R^3}e^{-ik\hx\cdot y}J(y)dy\times\hx\right)e^{ik\hx\cdot z}k^2dkds_{\hx}\\
&=\frac{1}{(2\pi)^{3}}\int_{\mathbb S^2}\int_0^{+\infty}\hx\times\left(\mathcal{F}[J](k\hx)\times\hx\right)e^{ik\hx\cdot z}k^2dkds_{\hx}\\
&=\frac{1}{(2\pi)^{3}}\int_{\mathbb S^2}\int_0^{+\infty}\left\{\mathcal{F}[J](k\hx)-\hx[\mathcal{F}[J](k\hx)\cdot\hx]\right\}e^{ik\hx\cdot z}k^2dkds_{\hx}
,\ z\in\mathbb R^3.
\end{align*}
By the assumption ${\rm div}J=0$, we have
\ben
\mathcal{F}[J](\xi)\cdot\xi=-i\mathcal{F}[{\rm div}J](\xi)=0.
\enn
Therefore, 
\begin{align*}
\mathcal{I}_E(z)
&=\frac{1}{(2\pi)^{3}}\int_{\mathbb S^2}\int_0^{+\infty}\mathcal{F}[J](k\hx)e^{ik\hx\cdot z}k^2dkds_{\hx}\\
&=\mathcal{F}^{-1}\mathcal{F}[J]\\
&=J(z),\ z\in\mathbb R^3.
\end{align*}
 The proof is complete.
\end{proof}

Note that, different from the acoustic and elastic source scattering problems, there exists non-radiating sources for the electromagnetic waves. Precisely, the electromagnetic far field patterns may vanish for the sources $J$ satisfying ${\rm div}J\neq0$. We refer to \cite{stability-elastic-electromagnetic,non-sou-EM} for more details on the non-radiating electromagnetic sources. Physically, $\rho:=\frac{1}{i\om}{\rm div}J$ is the charge density. We define
\be
    \label{indicator-h}
     &&\mathcal{I}_H(z):=\frac{\sqrt{\mu}}{2\pi^2}\int_{\mathbb S^2}\int_0^{+\infty}H_{\infty}(\hx,k)e^{ik\hx\cdot z}k^2dkds_{\hx},\ z\in\mathbb R^3,\\
    \label{indicator-rho}
    &&\mathcal{I}_{\rho}(z):=\frac{1}{(2\pi)^3}\int_{\mathbb S^2}\int_0^{+\infty}\left(\omega\mathcal{F}[\rho](k\hx)-4\pi i\sqrt{\varepsilon}E_{\infty}(\hx,k)\right)e^{ik\hx\cdot z}kdkds_{\hx},\ z\in\mathbb R^3.
\en
Following the arguments in the proof of Theorem \ref{thm-e}, we obtain the following theorem. To avoid repetition, we omit the proof.
\begin{theorem}
For $J\in\left(H^1(\mathbb R^3)\right)^3$, we have
\ben
\mathcal{I}_H={\rm \curl}J.
\enn
If we know the charge density $\rho\in L^2(\mathbb R^3)$ in advance, we have
\ben
\mathcal{I}_{\rho}=J.
\enn
\end{theorem}
We finally remark that, without any a priori information on the charge density,  we can always reconstruct ${\rm \curl}J$ from the magnetic far field patterns.

\section{Stability estimates for the discrete indicators}
Definitely, the theorems in the previous two sections show that the indicators $\mathcal{I}_f,\mathcal{I}_p,\mathcal{I}_s,\mathcal{I}_E$ and $\mathcal{I}_H$ can be used to determine the full or partial information of unknown sources.
In practice, the far field patterns are taken for finitely many observation directions and frequencies. Therefore, we have to consider the indicators in the form of a finite sum. In this section, we derive the stability analyses for such indicators.

We begin with the elastic source reconstructions in $\R^2$. 
The observation direction set is defined by
\ben
\Theta_L:=\left\{\Big(\cos\frac{2l \pi}{L}, \sin\frac{2l \pi}{L}\Big)^T\, \Big{|}\, l=0, 1, \cdots, L-1\right\}.
\enn\\
We take the circular frequencies
\ben
\omega_m=m\Delta \omega,\quad m=1,2,\cdots, \Lambda.
\enn
Given a source function $S$, we compute  synthetic approximation $u^{\infty}(\hx_l,\omega_m)$ from \eqref{exsin2D} by the trapezoidal rule. We
further perturb this data by random noise as follow
 \begin{align}
 \label{dataerror}
     u_t^{\infty,\delta}(\hx_l,\omega_m)=u_t^{\infty}(\hx_l,\omega_m)\Big(1+\delta\big[N(0,1)+N(0,1)i\big]\Big),\quad t=p,s,
 \end{align}
 where $N(0,1)$ is a normal distribution with mean zero and standard derivation one.  We define $k_{p,m}:=\frac{\omega_m}{\sqrt{\lambda+2\mu}}$ and $k_{s,m}:=\frac{\omega_m}{\sqrt{\mu}}$, then we introduce an indicator function $I_f$ as follow,
\begin{align}\label{IF}
    I_f(z):=\frac{\Delta\omega}{2\pi L}\sum\limits_{\hx\in\Theta_L}\sum^{\Lambda}_{m=1}\omega_m\left[u^{\infty,\delta}_p(\hat{x},\omega_m)\frac{e^{ik_{p,m}\hx\cdot z}}{\lambda+2\mu}
    +u_s^{\infty,\delta}(\hat{x},\omega_m)\frac{e^{ik_{s,m}\hx\cdot z}}{\mu}
    \right],\quad z\in\mathcal{G}.
\end{align}
Here, $\mathcal{G}$ is a bounded domain containing the source support $\Om$. It is easy to verify that $I_f$ is an approximation of $\mathcal{I}_f$ given by \eqref{indicator-f} using the right rectangle rule.. Therefore, $I_f$ is except to reconstruct the source.
Furthermore, if only partial far field patterns data are available, we design indicator functions $I_p$ and $I_s$ as follow
\begin{align}\label{IP}
    I_p(z):=\frac{i\Delta\omega}{2\pi L}\sum\limits_{\hx\in\Theta_L}\sum^{\Lambda}_{m=1}\omega_m^2 \left[\left(u^{\infty,\delta}_p(\hat{x},\omega_m)\cdot\hx\right)\frac{e^{ik_{p,m}\hx\cdot z}}{\left(\sqrt{\lambda+2\mu}\right)^3}
    \right],\quad z\in\mathcal{G}.
\end{align}
and
\begin{align}\label{IS}
    I_s(z):=\frac{i\Delta\omega}{2\pi L}\sum\limits_{\hx\in\Theta_L}\sum^{\Lambda}_{m=1}\omega_m^2\left[ \left(u^{\infty,\delta}_s(\hat{x},\omega_m)\cdot x^{\perp}\right)\frac{e^{ik_{s,m}\hx\cdot z}}{\left(\sqrt{\mu}\right)^3}
    \right],\quad z\in\mathcal{G}.
\end{align}

The following error estimates give the effectiveness of the above indicators.
\begin{theorem}\label{stab-2-ela}
    For $S\in\left(H^{2}(\mathbb R^2)\right)^2$, we have 
    \begin{align}
        \label{If-est-2}
        &| I_f(z)-S(z)|\leq C\left[\omega_{\Lambda}(\delta +\Delta\omega+L^{-1})+\omega^{-1}_{\Lambda}\right]\Vert S\Vert_{H^{2}},\quad z\in\mathcal{G},
    \end{align}
    For ${\rm div}S\in H^{2}(\mathbb R^2)$, we have
    \begin{align}
        \label{Ip-est-2}
        &|I_p(z)-{\rm div}S(z)|\leq C\left[\omega_{\Lambda}(\delta+\Delta\omega+L^{-1})+\omega_{\Lambda}^{-1}\right]\Vert{\rm div}S\Vert_{H^2},\ z\in\mathcal{G}.
    \end{align}
    For ${\rm div^{\perp}}S\in H^{2}(\mathbb R^2)$, we have
    \begin{align}
       \label{Is-est-2}
       &|I_s(z)-{\rm div}^{\perp}S(z)|\leq C\left[\omega_{\Lambda}(\delta+\Delta\omega+L^{-1})+\omega_{\Lambda}^{-1}\right]\Vert{\rm div}^{\perp}S\Vert_{H^2},\ z\in\mathcal{G}.
    \end{align}
    Here, the constant $C$ depends only on $\Omega, \mathcal{G}$ and the Lam\'e constants $\lambda$, $\mu$.
\end{theorem}
\begin{proof}
We only show \eqref{If-est-2} in details. The other two estimates \eqref{Ip-est-2}-\eqref{Is-est-2} can be proved similarly.

Straightforward calculations show that
\begin{align}\label{If-decomposition}
    4\pi^2I_f(z)=I_{f,p}(z)+I_{f,s}(z),\quad z\in\mathcal{G},
\end{align}
where
\begin{align*}
    I_{f,p}(z):=\sum_{\substack{|\xi_{m,l,p}|\leq k_{p,\Lambda}}}\left\{u_p^{\infty,\delta}(\hat{\xi}_{m,l,p},\sqrt{\lambda+2\mu}|\xi_{m,l,p}|) e^{i\xi_{m,l,p}\cdot z}|\xi_{m,l,p}|\frac{2\pi\Delta\omega}{\sqrt{\lambda+2\mu}L}\right\},\quad z\in\mathcal{G}
\end{align*}
and
\begin{align*}
    I_{f,s}(z):=\sum_{\substack{|\xi_{m,l,s}|\leq k_{s,\Lambda}}}\left\{u_s^{\infty,\delta}(\hat{\xi}_{m,l,s},\sqrt{\mu}|\xi_{m,l,s}|)e^{i\xi_{m,l,s}\cdot z}|\xi_{m,l,s}|\frac{2\pi\Delta\omega}{\sqrt{\mu}L}\right\},\quad z\in\mathcal{G}
\end{align*}
with
\begin{align*}
    \xi_{m,l,p}:=k_{p,m}e^{2il\pi/L}\quad\mbox{and}\quad \xi_{m,l,s}:=k_{s,m}e^{2il\pi/L},\quad m\in \{1,2,\cdots,\Lambda\},\, l\in\{0, 1,\cdots,L-1\}.
\end{align*}
Based on the Fourier transform and the inverse Fourier transform, the source $S$ can be represented as following form
\begin{align}\label{S-decomposition}
    4\pi^2S(z)
    =&\int_{\mathbb R^2}\mathcal{F}[S](\xi)e^{i\xi\cdot z}d\xi\cr
    =&\int_{\mathbb R^2}\mathbb{F}[S](\xi)e^{i\xi\cdot z}d\xi+\int_{\mathbb R^2}\mathbb{F}^{\perp}[S](\xi)e^{i\xi\cdot z}d\xi\cr
    =&\int_{|\xi|\leq k_{p,\Lambda}}\mathbb{F}[S](\xi)e^{i\xi\cdot z}d\xi
    +\int_{|\xi|\geq k_{p,\Lambda}}\mathbb{F}[S](\xi)e^{i\xi\cdot z}d\xi\cr
    &+\int_{|\xi|\leq k_{s,\Lambda}}\mathbb{F}^{\perp}[S](\xi))e^{i\xi\cdot z}d\xi
    +\int_{|\xi|\geq k_{s,\Lambda}}\mathbb{F}^{\perp}[S](\xi)e^{i\xi\cdot z}d\xi,\quad z\in\mathcal{G},
\end{align}
where $\mathbb{F}^{\perp}[S](\xi):=\hat{\xi}^{\perp}\left(\mathcal{F}[S](\xi)\cdot \hat{\xi}^{\perp}\right)$ and $\mathbb{F}[S](\xi):=\hat{\xi}\left(\mathcal{F}[S](\xi)\cdot \hat{\xi}\right)$.
Then, from \eqref{If-decomposition} and \eqref{S-decomposition}, we have 
\begin{equation}\label{allerror}
        \begin{split}
            &4\pi^2|I_f(z)-S(z)|\\
            \leq&\left|\int_{|\xi|\leq k_{s,\Lambda}}\mathbb{F}^{\perp}[S](\xi)e^{i\xi\cdot z}d\xi-\sum_{\substack{|\xi_{m,l,s}|\leq k_{s,\Lambda}}}\left\{\mathbb{F}^{\perp}[S](\xi_{m,l,s})e^{i\xi_{m,l,s}\cdot z}|\xi_{m,l,s}|\frac{2\pi\Delta\omega}{\sqrt{\mu}L}\right\}\right|\\
            &+\left|\int_{|\xi|\leq k_{p,\Lambda}}\mathbb{F}[S](\xi)e^{i\xi\cdot z}d\xi-\sum_{\substack{|\xi_{m,l,p}|\leq k_{p,\Lambda}}}\left\{\mathbb{F}[S](\xi_{m,l,p})e^{i\xi_{m,l,p}\cdot z}|\xi_{m,l,p}|\frac{2\pi\Delta\omega}{\sqrt{\lambda+2\mu}L}\right\}\right|\\
       &+\left|\sum_{\substack{|\xi_{m,l,s}|\leq k_{s,\Lambda}}}\left\{\mathbb{F}^{\perp}[S](\xi_{m,l,s})e^{i\xi_{m,l,s}\cdot z}|\xi_{m,l,s}|\frac{2\pi\Delta\omega}{\sqrt{\mu}L}\right\} -I_{f,s}(z)\right|\\
       &+\left|\sum_{\substack{|\xi_{m,l,p}|\leq k_{p,\Lambda}}}\left\{\mathbb{F}[S](\xi_{m,l,p})e^{i\xi_{m,l,p}\cdot z}|\xi_{m,l,p}|\frac{2\pi\Delta\omega}{\sqrt{\lambda+2\mu}L}\right\}-I_{f,p}(z)\right|\\
       &+\left|\int_{|\xi|\geq k_{s,\Lambda}}\mathbb{F}^{\perp}[S](\xi)e^{i\xi\cdot z}d\xi\right|\\
       &+\left|\int_{|\xi|\geq k_{p,\Lambda}}\mathbb{F}[S](\xi)e^{i\xi\cdot z}d\xi\right|\\
       :=&E_1+E_2+E_3+E_4+E_5+E_6,\quad z\in\mathcal{G}.
        \end{split}
    \end{equation}
    Using polar coordinate transformation and differential mean value theorem, define $S_z(x):=S(x+z)$, we derive that
\be\label{firsterror}
    E_1(z)
    &&\leq\sum_{\substack{|\xi_{m,l,s}|\leq k_{s,\Lambda}}}\int_{F_{m,l}}\left|\frac{\partial(\rho\mathbb{F}^{\perp}[S_z])}{\partial\rho}(\xi_{m,l,s})\left(\rho-k_{s,m}\right)+\frac{\partial(\rho\mathbb{F}^{\perp}[S_z])}{\partial\theta}(\xi_{m,l,s})\left(\theta-\frac{2l\pi}{L}\right)\right| d\rho d\theta\cr
    &&\leq A\sum_{\substack{|\xi_{m,l,s}|\leq k_{s,\Lambda}}}\int_{F_{m,l}}\left|\rho-k_{s,m}\right|+\left|\theta-\frac{2l\pi}{L}\right| d\rho d\theta\cr
    &&= A\sum_{\substack{|\xi_{m,l,s}|\leq k_{s,\Lambda}}}\frac{1}{2}\left(\frac{2\pi}{L}\frac{\Delta\omega^2}{\mu}+\frac{\Delta\omega}{\sqrt{\mu}}\left(\frac{2\pi}{L}\right)^2\right)\cr
    &&= A\frac{L\Lambda}{2}\left(\frac{2\pi}{L}\frac{\Delta\omega^2}{\mu}+\frac{\Delta\omega}{\sqrt{\mu}}\left(\frac{2\pi}{L}\right)^2\right)\cr
    &&\leq AC(\mu)\omega_{\Lambda}\left(\Delta\omega+L^{-1}\right),\quad z\in\mathcal{G},    
\en
where
\ben
&A:=\Vert\mathcal{F}[S_z](\xi)\Vert_{L^{\infty}}+\sum\limits_{\mathbbm{i,j}=1,2}\Vert|\xi|^{\mathbbm{i}}\pa_{\mathbbm{j}}\mathcal{F}[S_z](\xi)\Vert_{L^{\infty}},\\
&F_{m,l}:=\left\{\xi\in\mathbb R^2\Big|{\rm arg}(\xi)\in\left(\frac{2(l-1)\pi}{L},\frac{2l\pi}{L}\right),(m-1)\Delta\omega<\sqrt{\mu}|\xi|<m\Delta\omega\right\}.
\enn
Using the H\"{o}lder inequality, for $\forall\ \xi\in\mathbb R^2$ and $\mathbbm{j},\mathbbm{i}=1,2$, we have
\begin{align*}
     |\mathcal{F}[S_z](\xi)|
    &=\left|\int_{\mathbb R^2}S_z(x)e^{i\xi\cdot x}dx\right|\leq\int_{\Omega}|S_z(x)|dx\leq C(\Om)\Vert S\Vert_{L^2},\\
    |\xi_{\mathbbm{i}}\partial_{\mathbbm{j}}\mathcal{F}[S_z](\xi)|
    &=|\mathcal{F}[\partial_{\mathbbm{i}}(x_{\mathbbm{j}}S_z)](\xi)|\cr &=\left|\int_{\mathbb R^2}\partial_{\mathbbm{i}}(x_{\mathbbm{j}}S_z(x))e^{i\xi\cdot x}dx\right|\cr
    &\leq\int_{\Omega}|x_{\mathbbm{j}}\partial_{\mathbbm{i}}S_z(x)|+|S_z(x)|dx\cr
    &\leq C(\Omega,\mathcal{G})\Vert S\Vert_{H^1}.
\end{align*}
and 
\begin{align*}
     |\xi^2_{\mathbbm{i}}\partial_{\mathbbm{j}}\mathcal{F}[S_z](\xi)|
    &=|\mathcal{F}[\partial_{\mathbbm{ii}}(x_{\mathbbm{j}}S_z)](\xi)|\cr 
    &=\left|\int_{\mathbb R^2}\partial_{\mathbbm{ii}}(x_{\mathbbm{j}}S_z(x))e^{i\xi\cdot x}dx\right|\cr
    &\leq\int_{\Omega}|x_{\mathbbm{j}}\partial_{\mathbbm{ii}}S_z(x)|+2\delta_{\mathbbm{ij}}\partial_{\mathbbm{i}}S_z(x)|dx\cr
    &\leq C(\Omega,\mathcal{G})\Vert S\Vert_{H^2}.
\end{align*}
We now summarize the three inequalities to obtain $A\leq C(\Omega,\mathcal{G})\Vert S\Vert_{H^2}$ and then insert this into \eqref{firsterror} to conclude that
\begin{align}\label{firsterror-s}
    E_1(z)\leq C(\Omega,\mathcal{G},\mu)\omega_{\Lambda}(\Delta\omega+L^{-1})\Vert S\Vert_{H^2}.
\end{align}
Similarly,
\begin{align}\label{firsterror-p}
    E_2(z)\leq C(\Omega,\mathcal{G},\lambda,\mu)\omega_{\Lambda}(\Delta\omega+L^{-1})\Vert S\Vert_{H^2}.
\end{align}
Furthermore, using again H\"{o}lder inequality, we see from \eqref{exsin2D} and \eqref{dataerror} that
\begin{align}\label{seconderror-s}
   E_3(z)
   &\leq \sum_{\substack{|\xi_{m,l,s}|\leq k_{s,\Lambda}}} \left|\mathbb{F}^{\perp}[S](\xi_{m,l,s})|\xi_{m,l,s}|-u_s^{\infty,\delta}(\hat{\xi}_{m,l,s},|\xi_{m,l,s}|)|\xi_{m,l,s}|\right|\frac{2\pi\Delta\omega}{\sqrt{\mu}L}\cr
    &\leq \sum_{\substack{|\xi_{m,l,s}|\leq k_{s,\Lambda}}}C\delta  \Vert|\xi|\mathcal{F}[S](\xi)\Vert_{L^{\infty}}\frac{2\pi\Delta\omega}{\sqrt{\mu}L}\cr
    &\leq C(\Omega,\mu)\omega_{\Lambda}\delta\Vert S\Vert_{H^1},\quad z\in\mathcal{G}
\end{align}
and similarly
\begin{align}\label{secondererror-p}
   E_4(z)\leq C(\Omega,\mu,\lambda)\omega_{\Lambda}\delta\Vert S\Vert_{H^1},\quad z\in\mathcal{G}.
\end{align}
Furthermore, by the H\"{o}lder inequality, we have   
\begin{equation}
    \label{thirderror-s}
    \begin{split}
        E_5(z)&\leq\int_{|\xi|\geq k_{s,\Lambda}}\left|\mathcal{F}[S](\xi)\right|d\xi\\
    &\leq\left(\int_{|\xi|\geq k_{s,\Lambda}}(1+|\xi|^2)^{-2}d\xi\right)^{1/2}\Vert S\Vert_{H^{2}}\\
    &\leq C(\mu)\omega_{\Lambda}^{-1}\Vert S\Vert_{H^{2}},
    \end{split}
\end{equation}
and similarly
\begin{align}
 \label{thirderror-p}
  E_6(z)\leq C(\mu,\lambda)\omega_{\Lambda}^{-1}\Vert S\Vert_{H^{2}}.
\end{align}
Finally, the estimate \eqref{stab-2-ela} follows by combining \eqref{allerror}-\eqref{thirderror-p}.
\end{proof}
Theorem \ref{stab-2-ela} indicates that the reconstruction resolution can be improved by taking larger $L$ and smaller $\Delta\om$. This will be verified by the later numerical examples in the next section. Note that $\mathcal{F}[S]$ is smooth since $x^{\beta}S(x)\in\left(L^1(\mathbb R^2)\right)^2$ for any multiindex $\beta$. Therefore, the integral in \eqref{indicator-f} can be approximated using a high accuracy formula, then the power of $\Delta\omega$ and $L^{-1}$ in stability estimates can be improved.

In $\R^3$, the observation direction set is chosen by Fibonacci lattices:
\ben
\hx_l=(x_{(l,1)},x_{(l,2)},x_{(l,3)})^T,\quad l=1,2,\cdots,L
\enn
with
\begin{align*}
    x_{(l,3)}=1-\frac{2l}{L},\,
    x_{(l,1)}=\sqrt{1-x_{(l,3)}^2}\cos\left((\sqrt{5}-1)\pi l\right)\,\mbox{and}\,
    x_{(l,2)}=\sqrt{1-x_{(l,3)}^2}\sin\left((\sqrt{5}-1)\pi l\right).
\end{align*}
Note that the Fibonacci lattices are nearly uniformly distributed on the sphere $\mathbb S^2$, thus we approximate the spherical area element $ds_{\hx}$ by ${4\pi}/{L}$. The analogous stability estimates of Theorem \ref{stab-2-ela} in $\mathbb R^3$ can be obtained similarly.
For electromagnetic sources, we take the wave numbers
\ben
k_m=m\Delta k,\ m=1,2,\cdots,\Lambda.
\enn
Given a source term $J$, we also compute  synthetic approximation $E_{\infty}(\hx_l,\omega_m)$ from \eqref{ME-exp} by the trapezoidal rule. We
further perturb this data by random noise as follows,
 \begin{align*}
     E_{\infty,\delta}(\hx_l,k_m)&=\Big(1+\delta\big[N(0,1)+N(0,1)i\big]\Big)E_{\infty}(\hx_l,k_m),\\
    H_{\infty,\delta}(\hx_l,k_m)&=\Big(1+\delta\big[N(0,1)+N(0,1)i\big]\Big)\sqrt{\frac{\varepsilon}{\mu}}\hx_l\times E_{\infty}(\hx_l,k_m).
 \end{align*}
Then, we define
\begin{align}\label{IE}
 I_E(z):=\frac{-2 i\sqrt{\varepsilon}\Delta k}{\pi L}\sum\limits_{\hx\in\Theta_L}\sum_{m=1}^{\Lambda}\left[E_{\infty,\delta}(\hx,k_m)e^{ik_m\hx\cdot z}k_m\right],\ z\in\mathcal{G},
\end{align}
and
\begin{align}
 I_H(z):=\frac{2\sqrt{\mu}\Delta k}{\pi L}\sum\limits_{\hx\in\Theta_L}\sum_{m=1}^{\Lambda}\left[H_{\infty,\delta}(\hx,k_m)e^{ik_m\hx\cdot z}k_m^2\right],\ z\in\mathcal{G}.
\end{align}
Similar to the elastic wave case, we have the following stability estimate. To avoid repetition, we omit the proof.
\begin{theorem}\label{stab-3-maxwell}
    For $J\in\left(H^{3}(\mathbb R^3)\right)^3$ with 
   ${\rm div}J=0$, we have
    \begin{align*}
        |I_E(z)-J(z)|\leq C\left[k^2_{\Lambda}\Big(\delta +\Delta k+L^{-1}\Big)+k^{-3/2}_{\Lambda}\right]\Vert J\Vert_{H^{3}}\quad \mbox{for all}\,\ z\in\mathcal{G}.
    \end{align*}
   For ${\rm curl}J\in\left(H^{3}(\mathbb R^3)\right)^3$, we have 
    \begin{align*}
        |I_H(z)-{\rm curl}J(z)|\leq C\left[k^2_{\Lambda}\Big(\delta +\Delta k+L^{-1}\Big)+k^{-3/2}_{\Lambda}\right]\Vert {\rm curl}J\Vert_{H^{3}}\quad \mbox{for all}\,\ z\in\mathcal{G}.
    \end{align*}
    Here, the constant $C$ depends only on $\varepsilon$, $\mu$, $\Omega$ and $\mathcal{G}$.
\end{theorem}

\section{Numerical simulations}
Now we turn to present four groups of numerical examples to illustrate the effectiveness and robustness of the proposed quantitative sampling method.
Without loss of generality, the numerical simulations for elastic sources and  electromagnetic sources are presented in $\mathbb R^2$ and $\mathbb R^3$, respectively.

\subsection{Reconstruction for elastic wave source in $\mathbb R^2$.}
For elastic source reconstructions, we present two numerical examples by setting
\ben
\lambda = \mu = 1,\, \Delta \omega=0.5,\, \delta=0.3.
\enn
The grids are equally distributed on the rectangle $[-3,3]^2$ with sampling distance $0.01$. 

 \begin{itemize}
     \item \textbf{Example one: source $S\in\left(L^2(\mathbb R^2)\right)^2$
    .}
 \end{itemize}
In the first example, $S=(S_1,S_2)^T$ with compact support $\Omega$, where
\begin{align*}
    S_1(z)=|z|+5,\quad{\rm and}\quad S_2(z)=e^{0.1|z|^2}+4,\ \forall\ z\in\Omega.
\end{align*}
We refer to Figure \ref{source} for the true $S$ and $\Omega$.
\begin{figure}[htbp]
   \centering
    \begin{tabular}{cc}
        \subfigure[Image of $S_1$.]{
            \label{sourceS1}
        \includegraphics[width=0.3\textwidth]{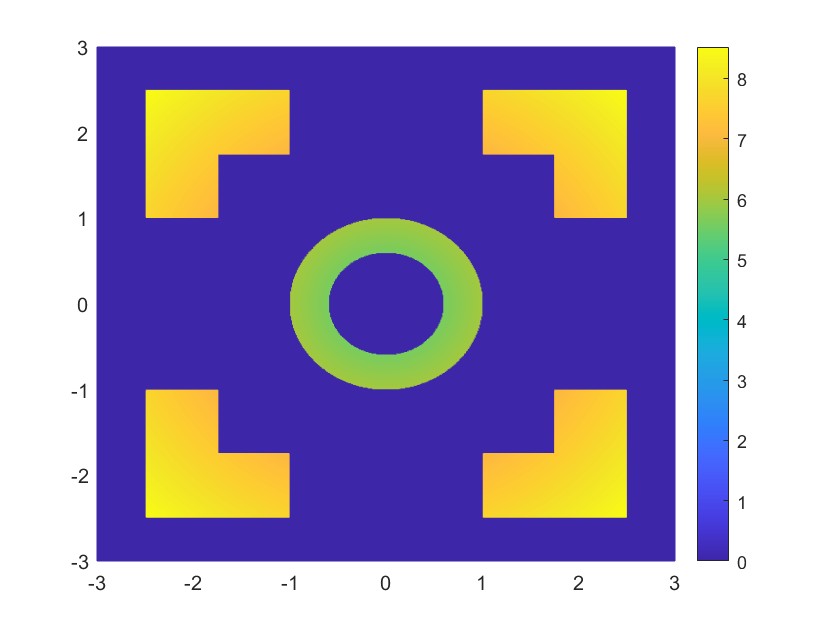}
        }&
        \subfigure[Image of $S_2$.]{
            \label{sourceS2}
            \includegraphics[width=0.3\textwidth]{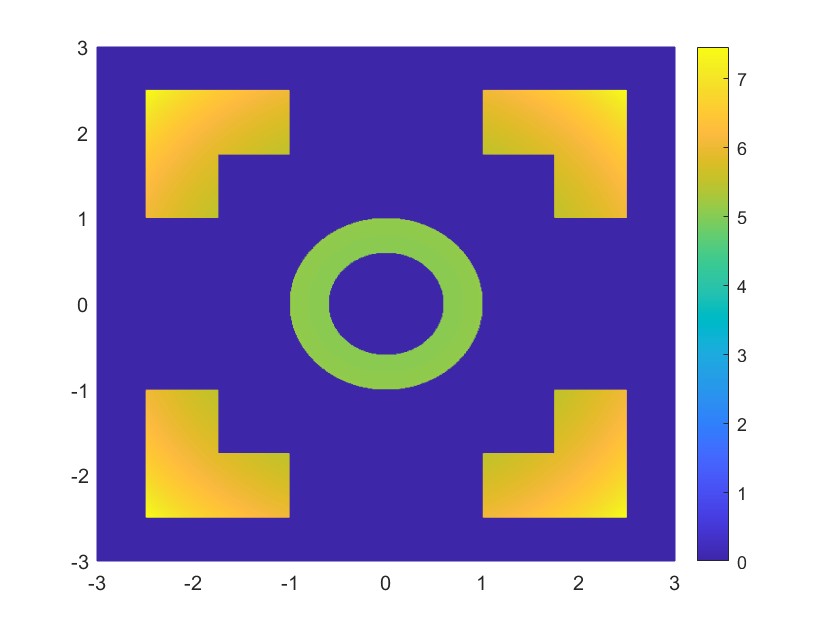}
        }
    \end{tabular}
    \caption{True source function $S$ and its support $\Omega$ in Example one.}
    \label{source}
\end{figure}
To further present the effectiveness of the indicator function $I_f$, we also show the difference of $S(z)$ and $I_f(z)$ by the following indicator
\be\label{Ieps}
\left(I^{\epsilon}_F(z)\right)_{(i)}:=\left\{
    \begin{array}{ll}
        1, &  {\rm if}\, |\left(S(z)-I_f(z)\right)_{(i)}|>\epsilon,\\
        0, & {\rm otherwise}.
    \end{array}
    \right.
\en
Here, $g_{(i)}$ denote the $i$-th component of an vector function $g$, $i=1,2$. Considering the $30\%$ relative noise in the far field patterns and the fact $S(z)>5,\ \forall\ z\in\Omega$, we set the threshold $\epsilon=5\times0.3=1.5$. 
\begin{figure}[htbp]
\centering
    \begin{tabular}{ccc}
        \subfigure[$L=31$.]{
            \label{L-31-S1}
            \includegraphics[width=0.3\textwidth]{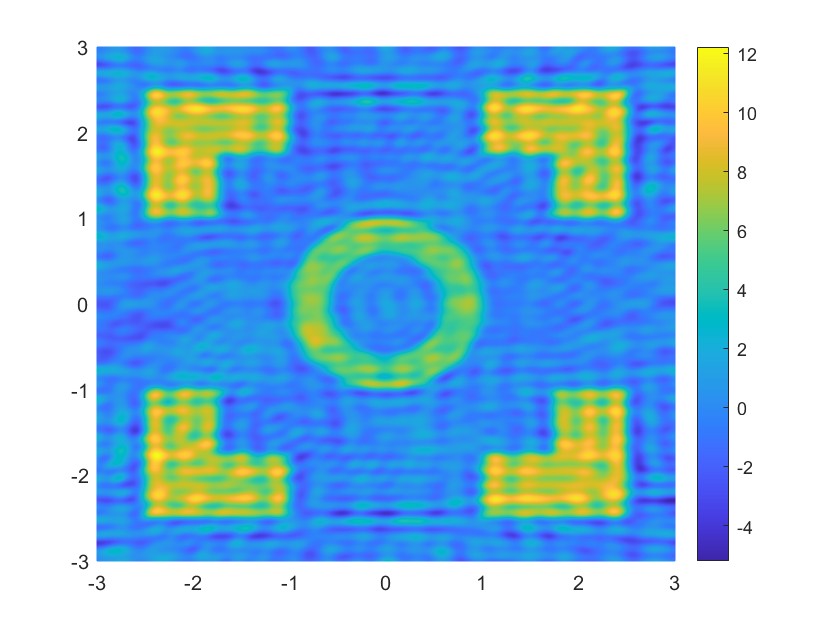}
        }&
        \subfigure[$L=51$.]{
            \label{L-51-S1}
            \includegraphics[width=0.3\textwidth]{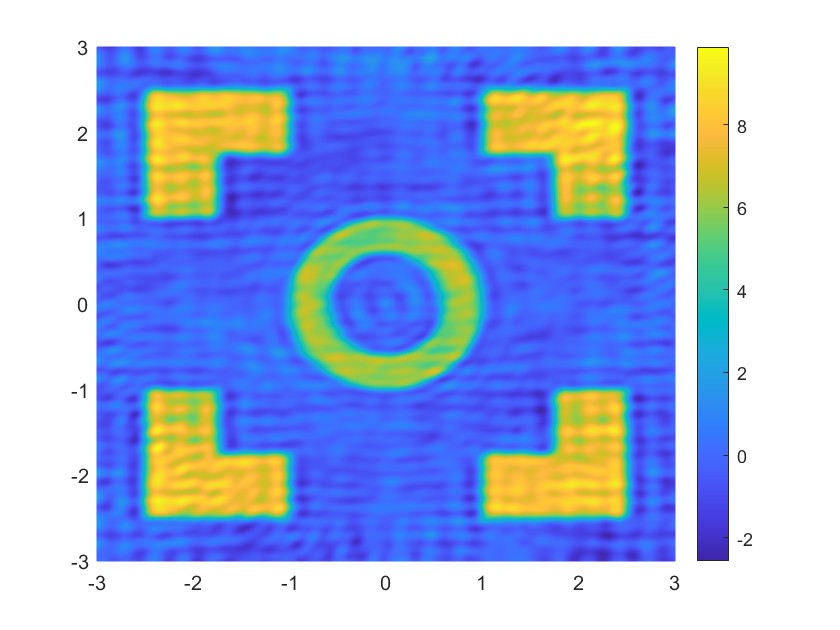}
        }&
         \subfigure[$L=71$.]{
            \label{L-71-S1}
            \includegraphics[width=0.3\textwidth]{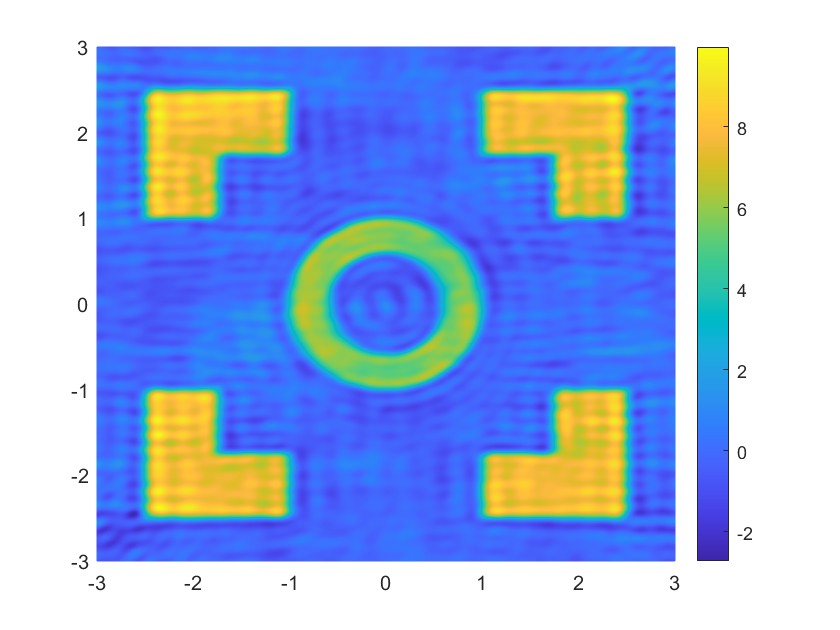}
        }\\
                \subfigure[$L=31$.]{
            \label{L-31-S1diff}
            \includegraphics[width=0.3\textwidth]{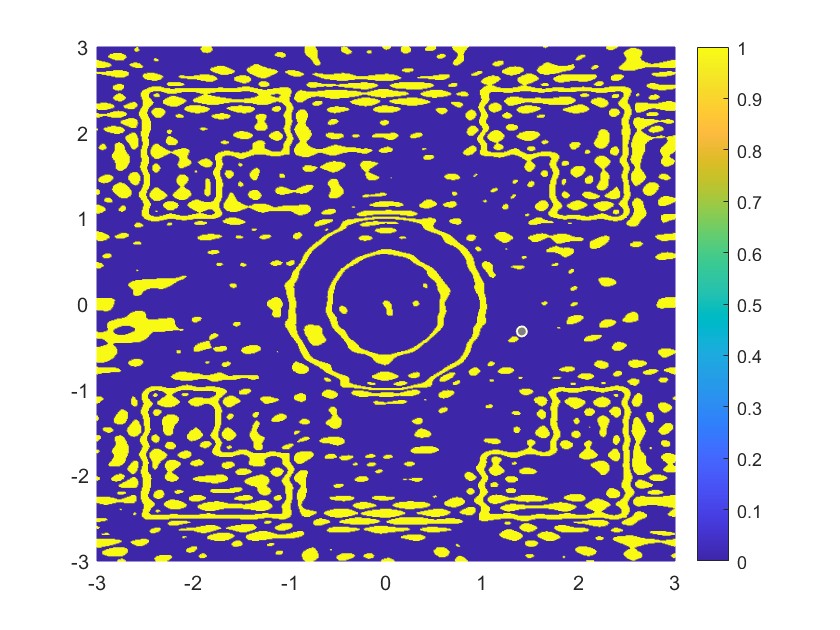}
        }&
        \subfigure[$L=51$.]{
            \label{L-51-S1diff}
            \includegraphics[width=0.3\textwidth]{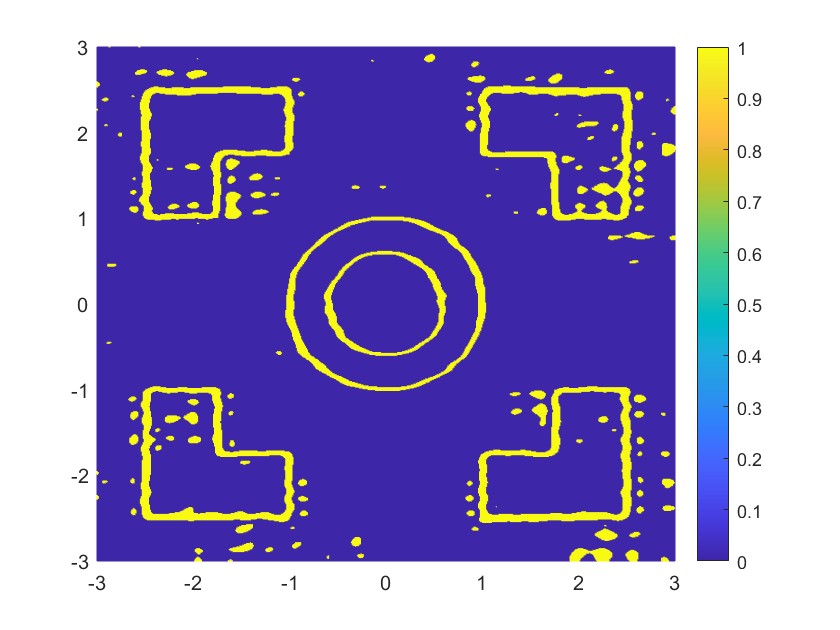}
        }&
         \subfigure[$L=71$.]{
            \label{L-71-S1diff}
            \includegraphics[width=0.3\textwidth]{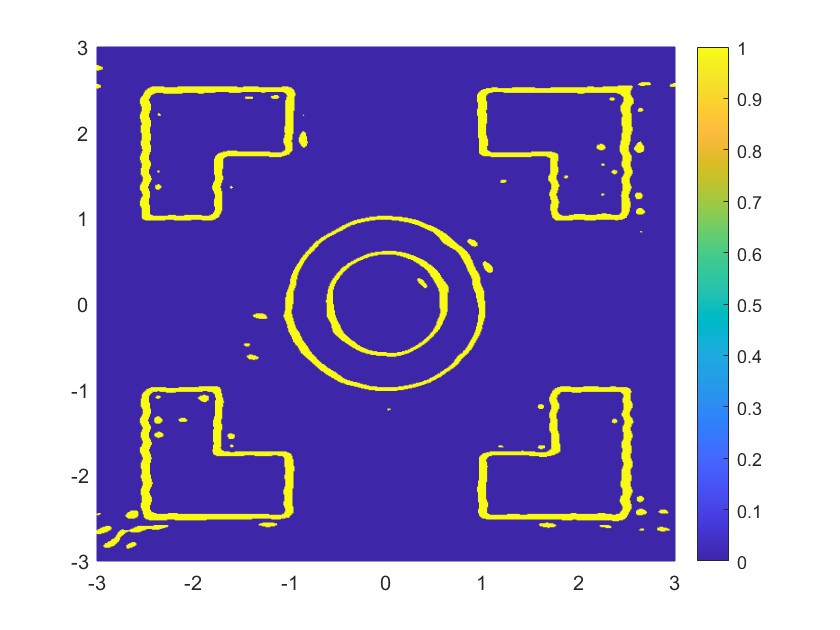}
        }
    \end{tabular}
    \caption{Reconstructions with $\omega_{\Lambda}=40$ and different $L$. Top row: reconstructions by plotting $(I_f)_{(1)}$. Bottom row: reconstructions by plotting $(I_f^{\epsilon})_{(1)}$.}
    \label{IF1-L}
\end{figure}
\begin{figure}[htbp]
   \centering
    \begin{tabular}{ccc}
        \subfigure[$\omega_{\Lambda}=30$.]{
            \label{lambda-30-S2}
        \includegraphics[width=0.3\textwidth]{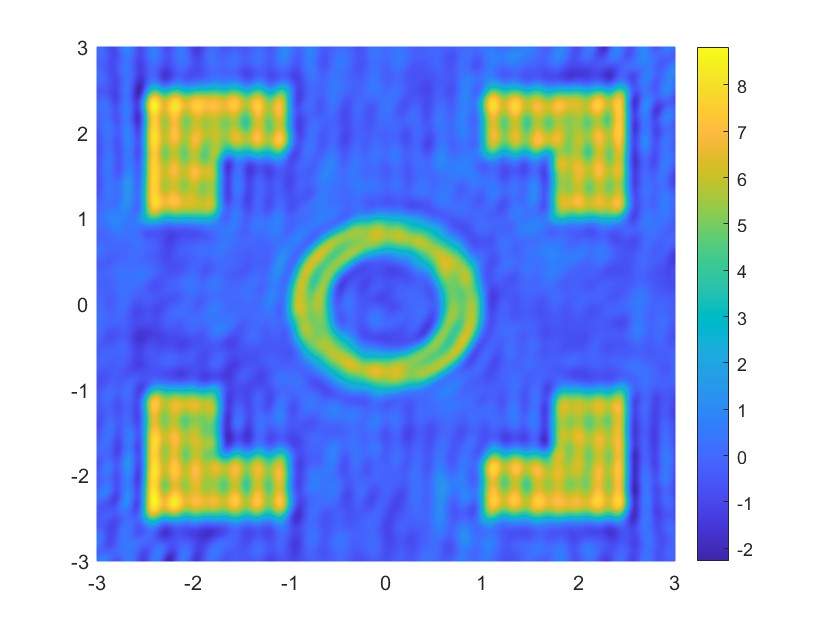}
        }&
        \subfigure[ $\omega_{\Lambda}=40$.]{
            \label{lambda-40-S2}
            \includegraphics[width=0.3\textwidth]{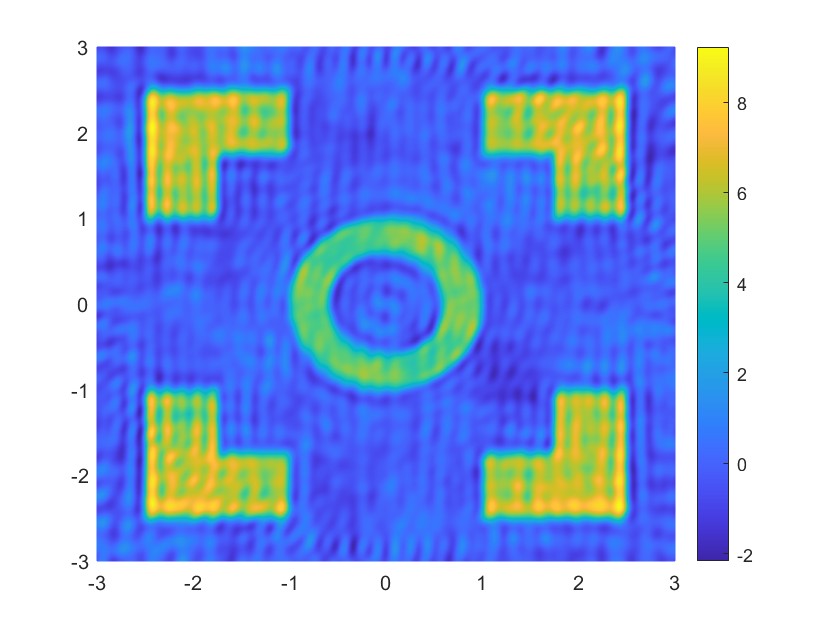}
        }&
         \subfigure[$\omega_{\Lambda}=50$.]{
            \label{lambda-50-S2}
            \includegraphics[width=0.3\textwidth]{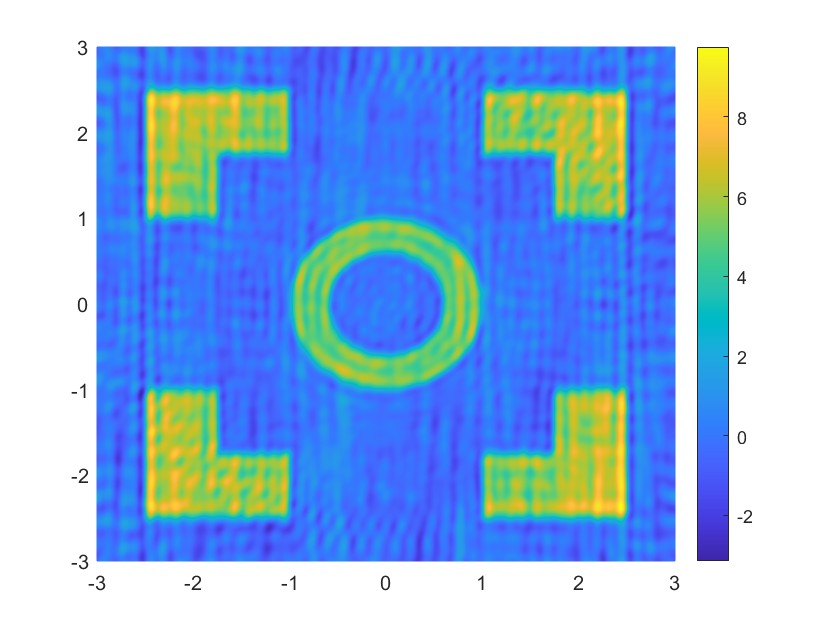}
        }\\
        \subfigure[$\omega_{\Lambda}=30$.]{
            \label{lambda-30-S2diff}
            \includegraphics[width=0.3\textwidth]{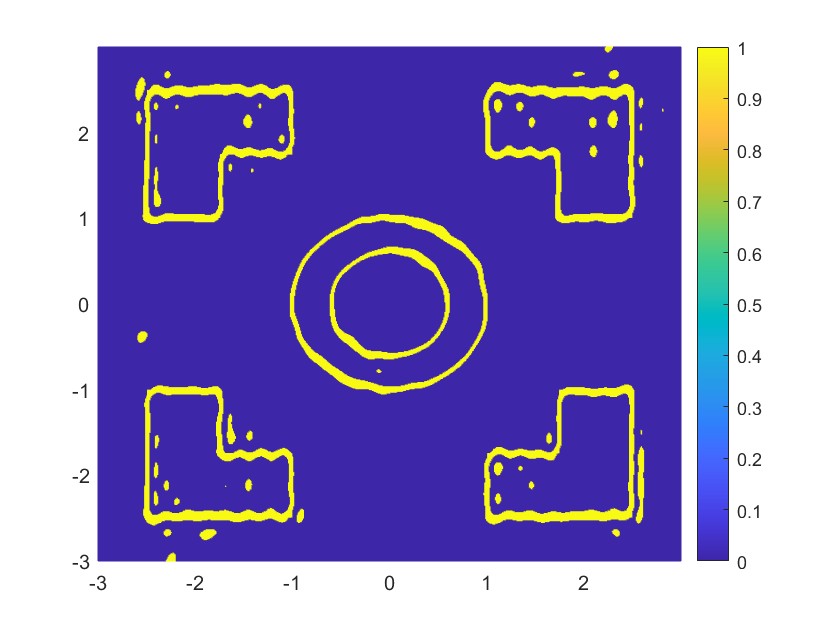}
        }&
        \subfigure[ $\omega_{\Lambda}=40$.]{
            \label{lambda-40-S2diff}
            \includegraphics[width=0.3\textwidth]{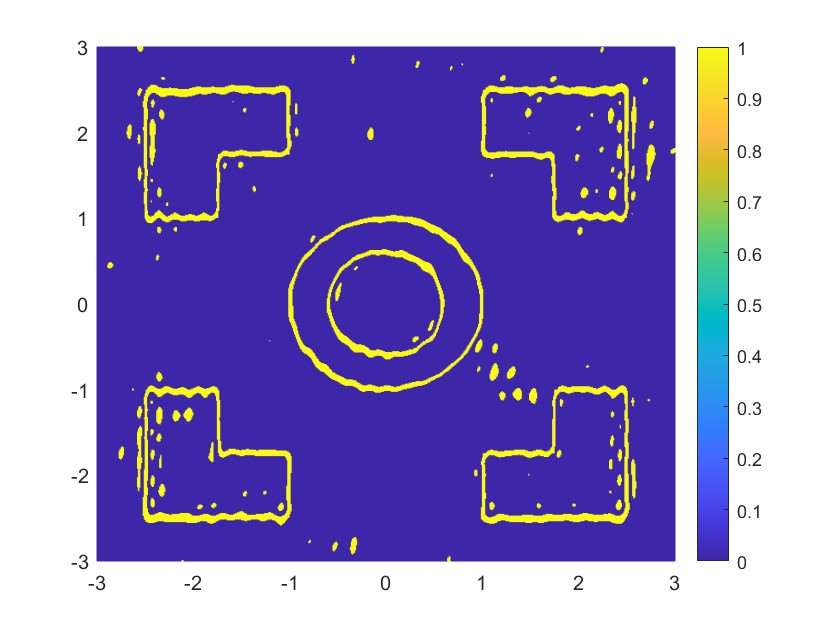}
        }&
         \subfigure[ $\omega_{\Lambda}=50$.]{
            \label{lambda-50-S2diff}
            \includegraphics[width=0.3\textwidth]{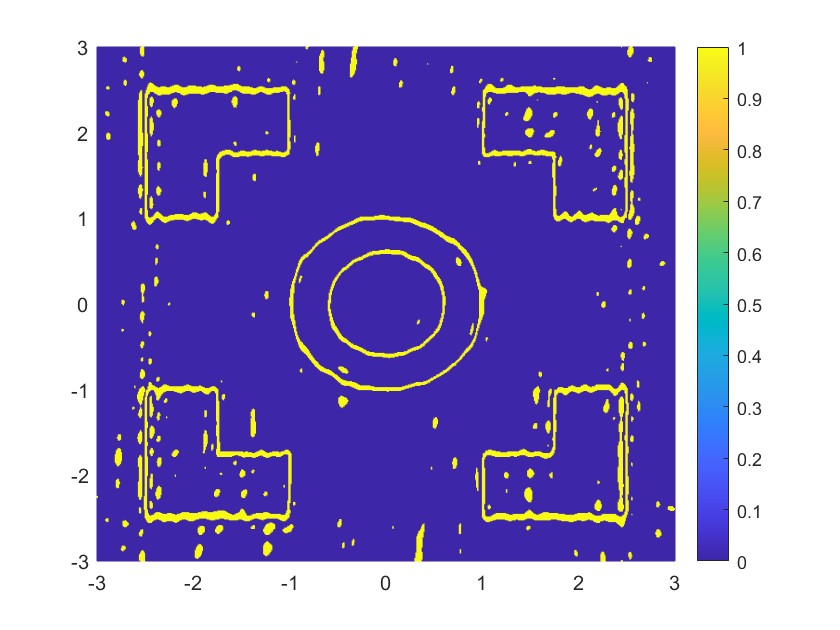}
        }
    \end{tabular}
    \caption{Reconstructions with $L=51$ and different $\omega_{\Lambda}$. Top row: reconstructions by plotting $(I_f)_{(2)}$. Bottom row: reconstructions by plotting $(I_f^{\epsilon})_{(2)}$.}
    \label{IF2-lambda}
\end{figure}
Figure \ref{IF1-L} displays the reconstructions for $S_1$ with various $L$, while Figure \ref{IF2-lambda} presents the reconstructions for $S_2$ with different $\omega_{\Lambda}$. The quality of the reconstructions of $S$ can be improved by increasing $\omega_{\Lambda}$ and $L$. We want to remark that the reconstructions of the indicator function $I_f$ always have a significant error near $\partial\Omega$. This is reasonable because $S$ is discontinuous at $\partial \Omega$. We can divide $\mathbb R^2$ into 'the neighborhood (of $\partial \Omega$)' and 'other regions'. Generally speaking, as $\omega_{\Lambda}$ increases, the neighborhood becomes 'thinner', but when $L$ increases, the neighborhood remains almost unchanged. Meanwhile, as $L$ increases, the reconstructions are significantly improved in other regions, but when $\omega_{\Lambda}$ increases, the improvement of the reconstructions is not noticeable in other regions. This phenomenon implies that the multidirectional data are more important for reconstructing continuous parts of source $S$ while the high-frequency data are more important for reconstructing discontinuous parts of source $S$.

Based on Theorems \ref{thm-ps} and \ref{stab-2-ela}, $I_p$ and and $I_s$ are able to reconstruct ${\rm div}S$ and ${\rm div}^{\perp}S$, respectively. 
\begin{figure}[htbp]
   \centering
    \begin{tabular}{cc}
        \subfigure[Image of $I_p$.]{
            \label{sourcediv}
        \includegraphics[width=0.3\textwidth]{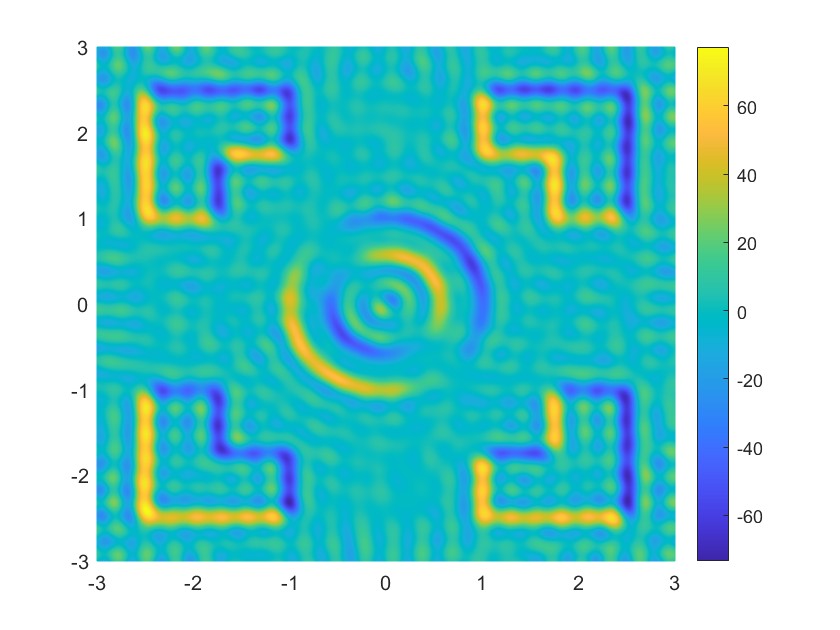}
        }&
        \subfigure[Image of $I_s$.]{
            \label{sourcedivperp}
            \includegraphics[width=0.3\textwidth]{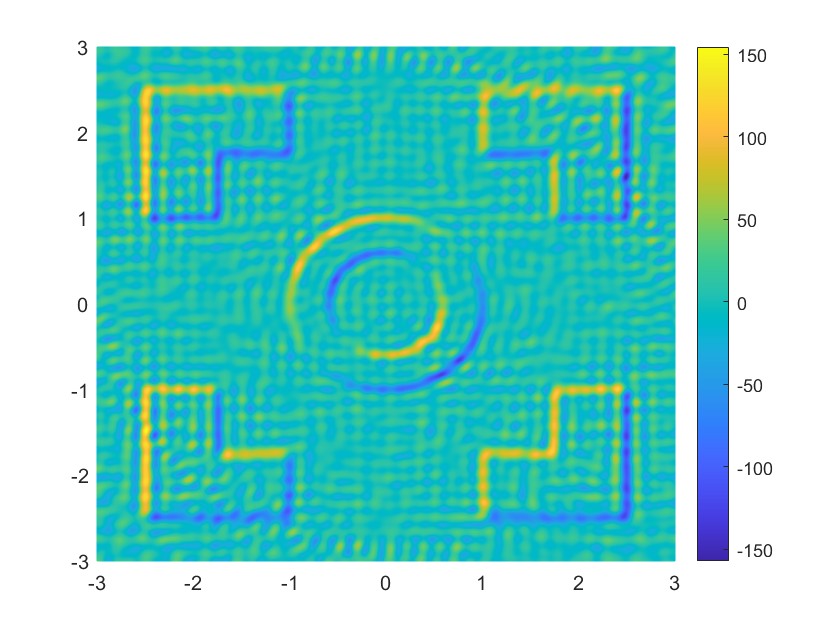}
        }
    \end{tabular}
    \caption{Reconstructions of $I_p$ and $I_s$ for Example one with $L=51$, $\omega_{\Lambda}=40$.}
    \label{source-div-divperp}
\end{figure}
 Figure \ref{source-div-divperp} show the reconstructions of $I_p$ and $I_s$ for the source $S$. Note that the source $S$ only belongs to
  $\left(H^1_{loc}\left(\mathbb R^2\backslash\partial\Omega\right)\right)^2$, it is understandable that $I_p$ and $I_s$ blow up near the $\partial\Omega$. This phenomenon also implies that, for the sources $S$ with non-trivial traces on $\partial\Omega$, $I_p$ and $I_s$ can be used to roughly recover $\partial\Omega$. Besides, we also observe that the indicator $I_s$ produces better reconstruction than $I_p$, which is in accordance with the fact that $k_s>k_p$. 
\begin{itemize}
    \item \textbf{Example two: smooth source.}
\end{itemize}
In the second example, we take $S=(S_1, S_2)^T$ with compact support $\Omega=\Omega_1\bigcup\Omega_2$, here for $\mathbbm{i}=1,2$,
\begin{equation*}
    S_\mathbbm{i}(z)=\left\{
    \begin{aligned}
        &(10|z+(-1)^\mathbbm{i}(1,1)^T|^4-11.6|z+(-1)^\mathbbm{i}(1,1)^T|^2+1.6)^2,\ &z\in\Omega_\mathbbm{i},\\
        &0,\ &z\in\mathbb R^2\backslash\Omega_{\mathbbm{i}},
    \end{aligned}
    \right.
\end{equation*}
with $\Omega_\mathbbm{i}:=\left\{z\in\mathbb R^2\big|0.4\leq|z+(-1)^\mathbbm{i}(1,1)^T|\leq 1\right\}$.
As shown in Figure \ref{ex2-if} (a,d), $\Om_1\cap\Om_2=\emptyset$. Moreover, such a source function has better regularity than the one in the first example.
To characterize the effectiveness of $I_f$, we define 
\begin{align*}
    e_F:=\frac{\Vert I_f-S\Vert_{L^2(\mathcal{G})}}{\Vert S\Vert_{L^2(\mathcal{G})}}.
\end{align*}
Table \ref{table-relativeerror} shows clearly that the relative error $e_F$ decreases with the increasing of $\omega_{\Lambda}$ or $L$. In particular, $e_F$ is much smaller than the data noise $\delta=0.3$. We also see from Table \ref{table-relativeerror-del} that $e_F$ indeed decreases for smaller $\Delta \om$, which is consistent with the stability estimates in previous section.
\begin{table}
     \centering
     \begin{tabular}{|c|c|c|c|c|}
     \hline
    \diagbox{$\omega_{\Lambda}$}{$e_F$}{$L$}&51&101&151&201\\
    \hline
    30&0.1584&0.1193&0.0973&0.0947\\
    \hline
    40&0.1494&0.1104&0.0899&0.0790\\
    \hline
    50&0.1459&0.1043&0.0835&0.0759\\
     \hline
     \end{tabular}
    \caption{Relative error of reconstructions using full  far field patterns with different $\omega_{\Lambda}$ and $L$. 
    Note the we have considered $30\%$ relative noise in the measurements.}
     \label{table-relativeerror}
\end{table} 
\begin{table}
     \centering
     \begin{tabular}{|c|c|c|c|c|}
     \hline
    \diagbox{{\footnotesize$\Delta\omega$}}{$e_F$}{$\delta$}&0.05&0.10&0.15&0.20\\
    \hline
    1/2&0.0511&0.0656&0.0830&0.1060\\
    \hline
    1/4&0.0455&0.0527&0.0659&0.0806\\
    \hline
    1/8&0.0439&0.0483&0.0553&0.0617\\
     \hline
     \end{tabular}
    \caption{Relative error of reconstructions using full  far field patterns with different $\delta$ and $\Delta\omega$. 
    Here we set $L=51$, $\omega_\Lambda=40$.}
     \label{table-relativeerror-del}
\end{table} 
Figure \ref{ex2-if} shows the reconstruction of the source function $S$ by $I_f$. 
Comparing Figures \ref{IF1-L}(d-f), \ref{IF2-lambda}(d-f) with Figure \ref{ex2-if}(c,f), we observe that the indicator $I_f$ gives a higher resolution reconstruction for the smoother sources. 
\begin{figure}[htbp]
   \centering
    \begin{tabular}{ccc}
        \subfigure[True source $S_1$.]{
            \label{sourceH1S1}
        \includegraphics[width=0.3\textwidth]{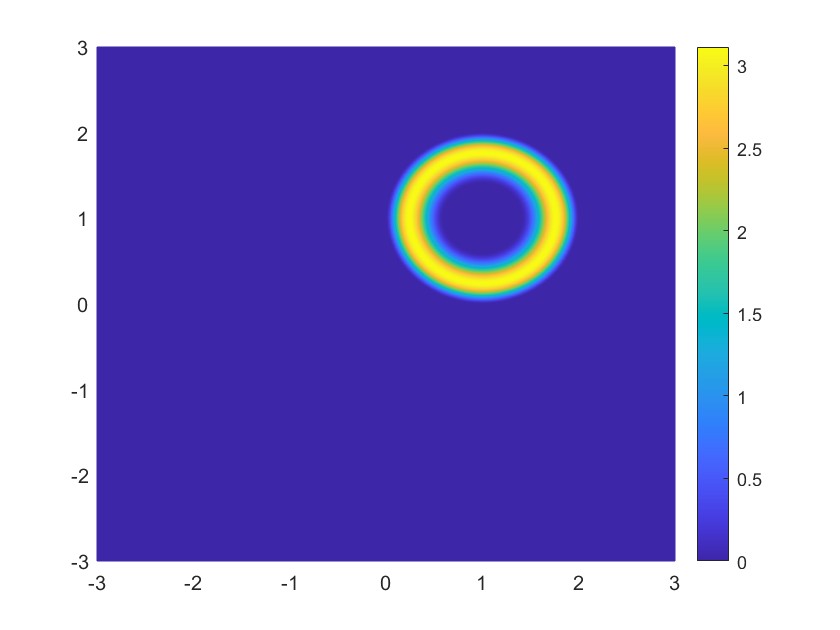}
        }&
         \subfigure[Image of $(I_f)_1$.]{
            \label{H1IF1}
        \includegraphics[width=0.3\textwidth]{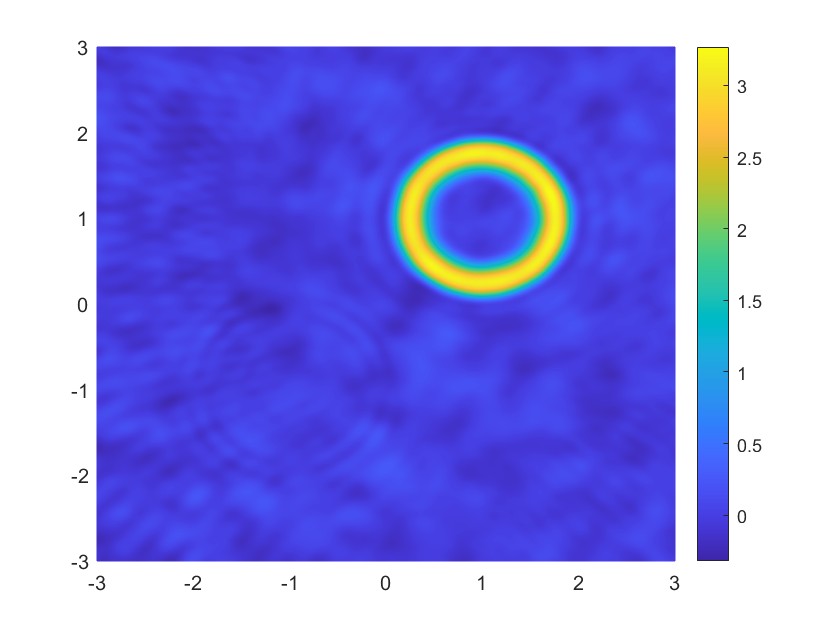}
        }&
        \subfigure[Image of $|S_1-(I_f)_1|$.]{
            \label{H1diff1}
            \includegraphics[width=0.3\textwidth]{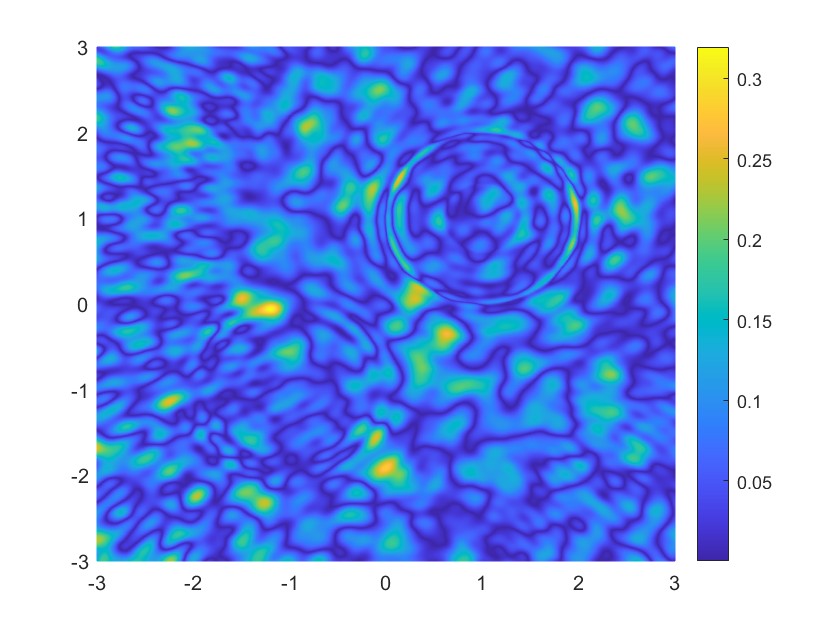}
        }\\
        \subfigure[True source $S_2$.]{
            \label{sourceH1S2}
            \includegraphics[width=0.3\textwidth]{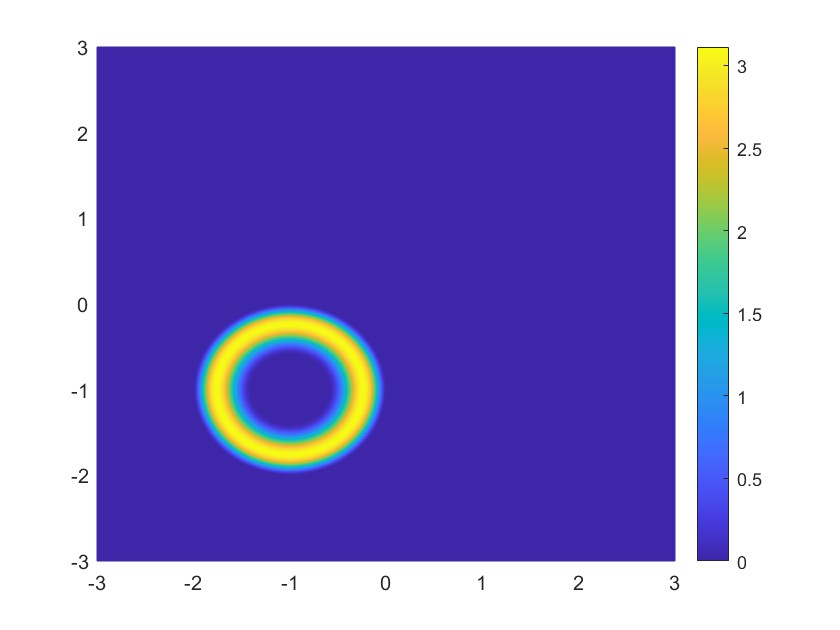}
        }&
        \subfigure[Image of $(I_f)_2$.]{
            \label{H1IF2}
            \includegraphics[width=0.3\textwidth]{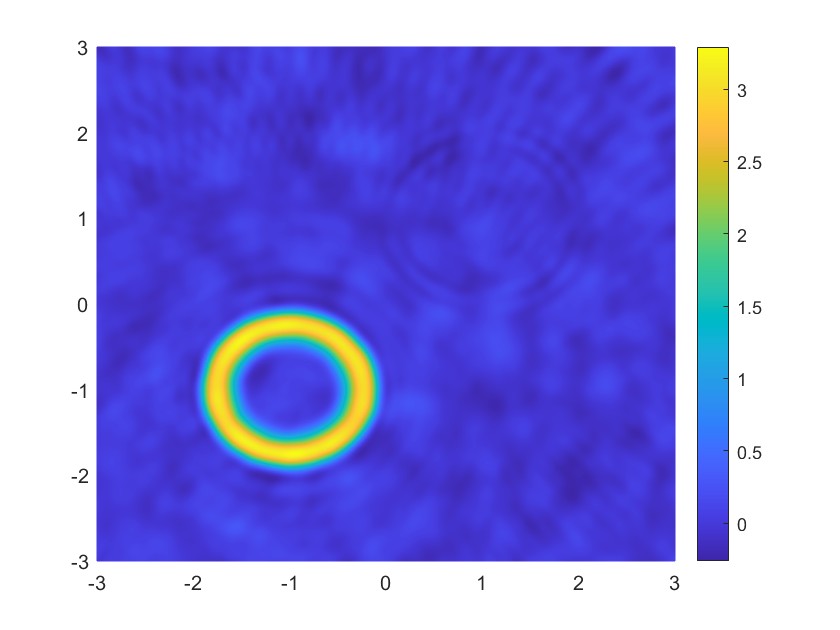}
        }&
        \subfigure[ Image of $|S_2-(I_f)_2|$.]{
            \label{H1diff2}
            \includegraphics[width=0.3\textwidth]{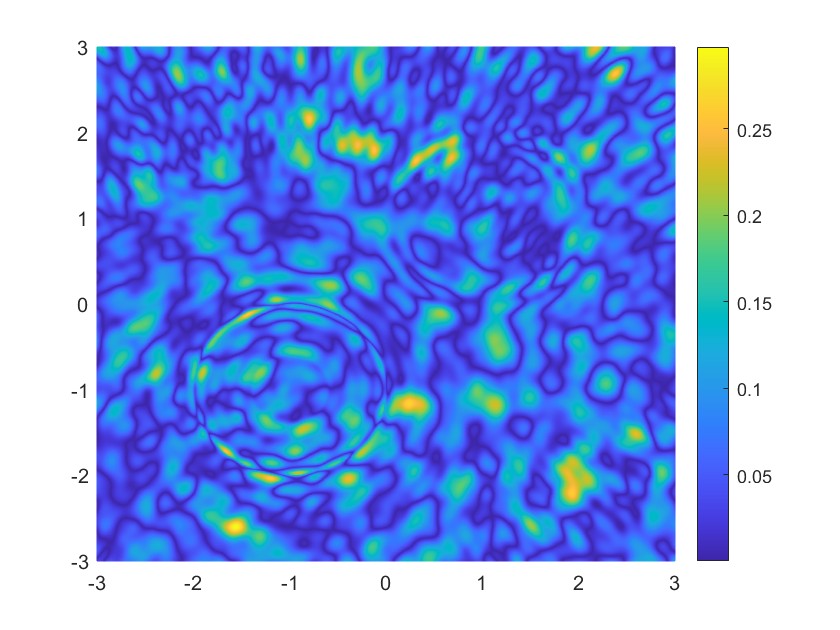}
        }
    \end{tabular}
    \caption{Reconstructions of $S_1$ and $S_2$ by $I_f$. We take $L=51$ and $\omega_{\Lambda}=40$.}
    \label{ex2-if}
\end{figure}
Furthermore, we investigate the effect of indicator function $I_p$ and $I_s$. Figure \ref{ex2-ipis} show that indicator functions $I_p$ and $I_s$ indeed give a quite robust and good reconstruction for ${\rm div} S$ and ${\rm div^{\perp}} S$, respectively. This is in accordance with Theorem \ref{thm-ps}.
Moreover, we also observe that, using the partial far field patterns, the source support $\Om$ is well recovered.
\begin{figure}[htbp]
   \centering
    \begin{tabular}{cc}
        \subfigure[True ${\rm div}S$.]{
            \label{sourcediv-2}
        \includegraphics[width=0.3\textwidth]{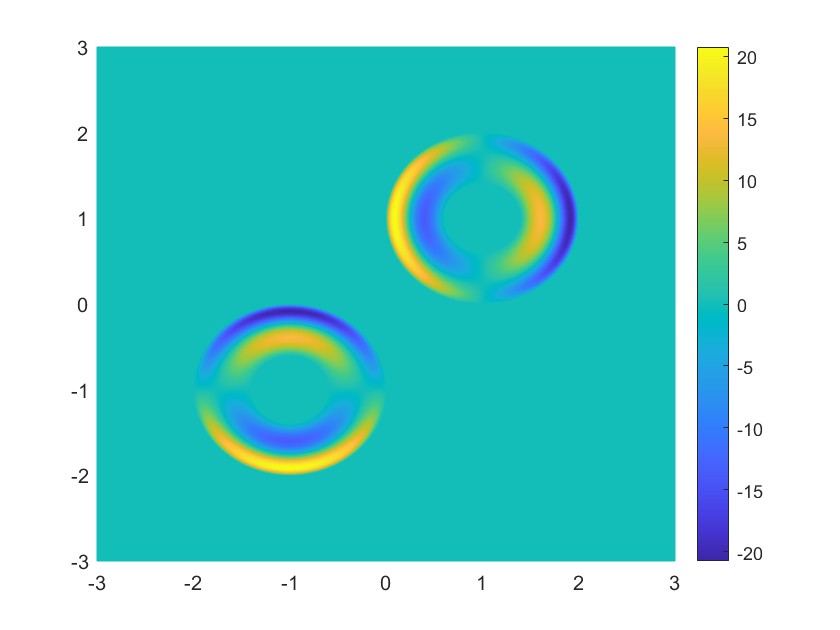}
        }&
        \subfigure[Image of $I_p$.]{
            \label{IP-2}
            \includegraphics[width=0.3\textwidth]{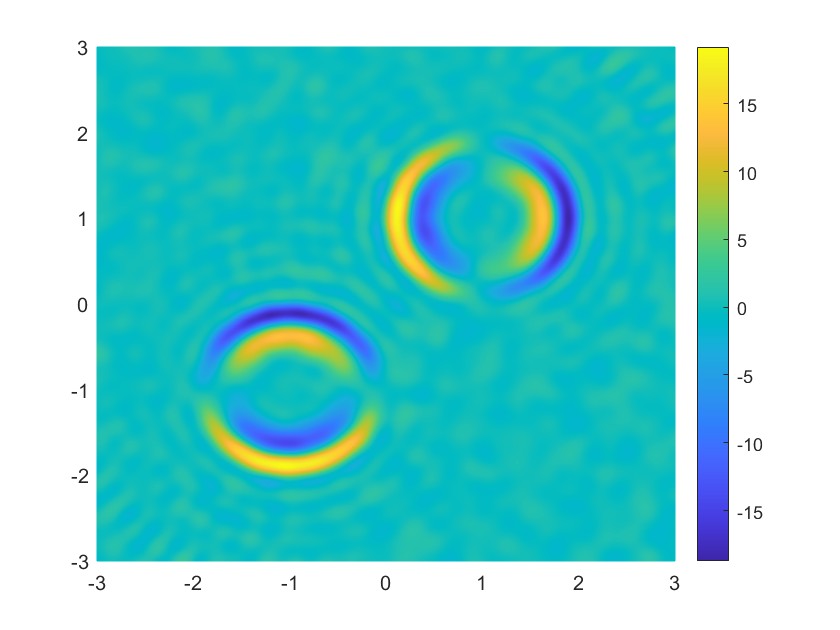}
        }\\
        \subfigure[True ${\rm div}^{\perp}S$.]{
            \label{sourcedivperp-2}
            \includegraphics[width=0.3\textwidth]{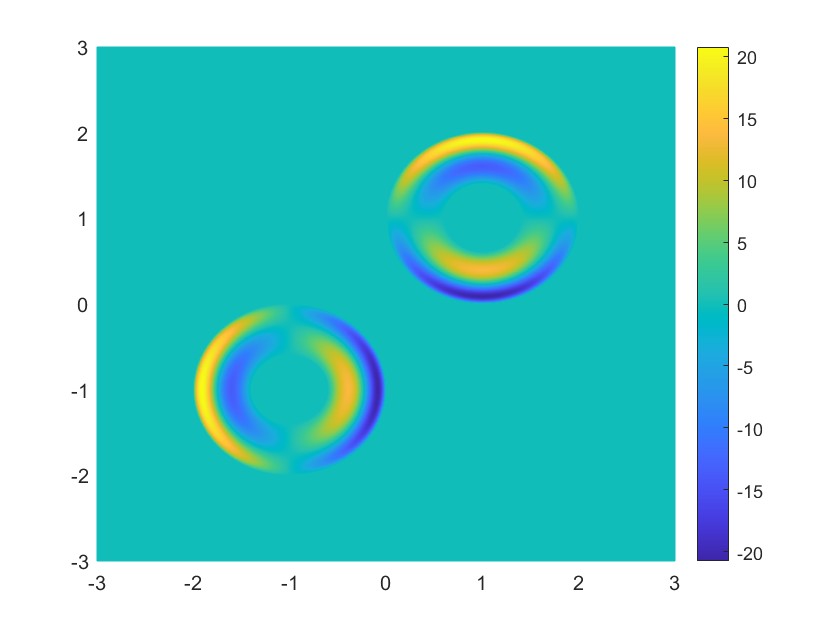}
        }&
        \subfigure[ Image of $I_s$.]{
            \label{IS-2}
            \includegraphics[width=0.3\textwidth]{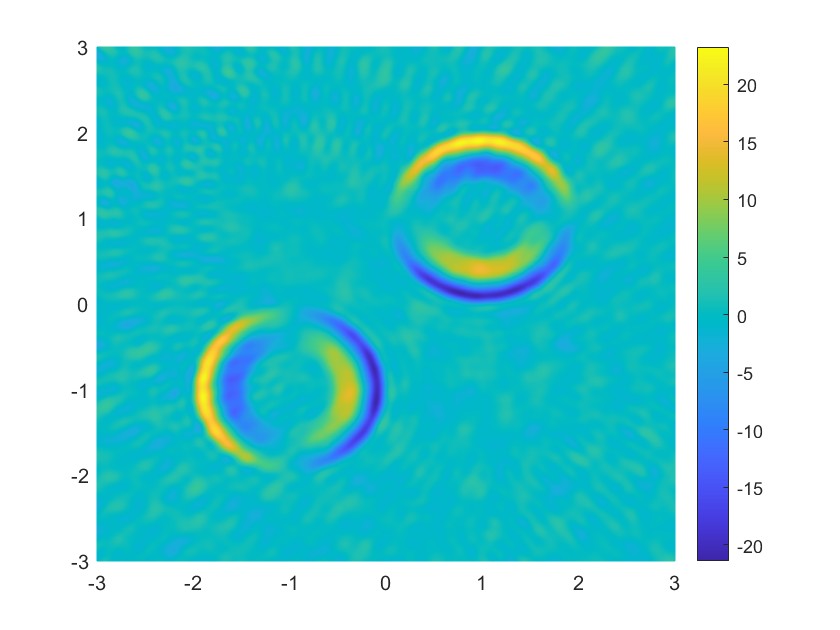}
        }
    \end{tabular}
    \caption{Reconstructions of ${\rm div}S$ and ${\rm div}^{\perp}S$ by $I_p$ and $I_s$, respectively. We take $L=51$ and $\omega_{\Lambda}=40$.}
    \label{ex2-ipis}
\end{figure}

\subsection{Reconstruction for electromagnetic sources in $\mathbb R^3$.}

For electromagnetic sources, we set 
 \ben
 \delta=0.1, \quad \Delta k=0.5.
 \enn
 The sampling grids are equally distributed in the cube $[-1,1]^3$ with sampling distance $0.01$. 
 
 \begin{itemize}
     \item  \textbf{Example three: the divergence free source $J\in\left(L^2(\mathbb R^3)\right)^3$.}
 \end{itemize}
In the third example, we set $J(z)=(4,4,4)^T+{\rm curl }(|z|^2,e^{|z|^2},1)^T,\quad z\in \Omega$ 
such that ${\rm div}J=0$. The source support $\Omega:={\rm supp}(J)=\Omega_1\cup\Omega_2$ with
\begin{align*}
    \Omega_1:=\left\{z=(z_1,z_2,z_3)^T\big|\ |z_1|\leq 0.5,\ |z_2|\leq 0.5,\ - 0.5\leq z_3\leq 0\right\},
\end{align*}
and
\begin{align*}
    \Omega_2:=\left\{z=(z_1,z_2,z_3)^T\big|0\leq z_1\leq 0.5,\ 0\leq z_2\leq 0.5,\ 0\leq z_3\leq 0.5\right\}.
\end{align*}

Three slices of the true source function $J$ are shown in Figure \ref{sourceJ}. Figure \ref{IndE-Omega_40} and \ref{IndE-L_151} show the reconstructions by varying $L$ and $k_\Lambda$, respectively. Obviously, the reconstruction quality can be improved by increasing $L$ or $k_{\Lambda}$. 
\begin{figure}[htbp]
   \centering
    \begin{tabular}{ccc}
        \subfigure[Slices of $J_1$ at $z_1=\pm 0.25$.]{
            \label{sourceJ1}
        \includegraphics[width=0.3\textwidth]{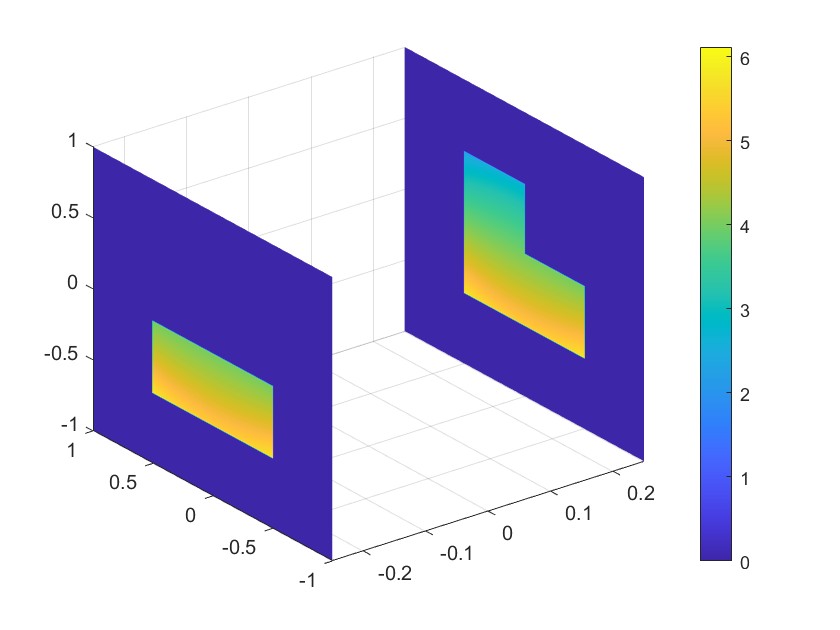}
        }&
        \subfigure[Slices of $J_2$ at $z_2=\pm 0.25$.]{
            \label{sourceJ2}
            \includegraphics[width=0.3\textwidth]{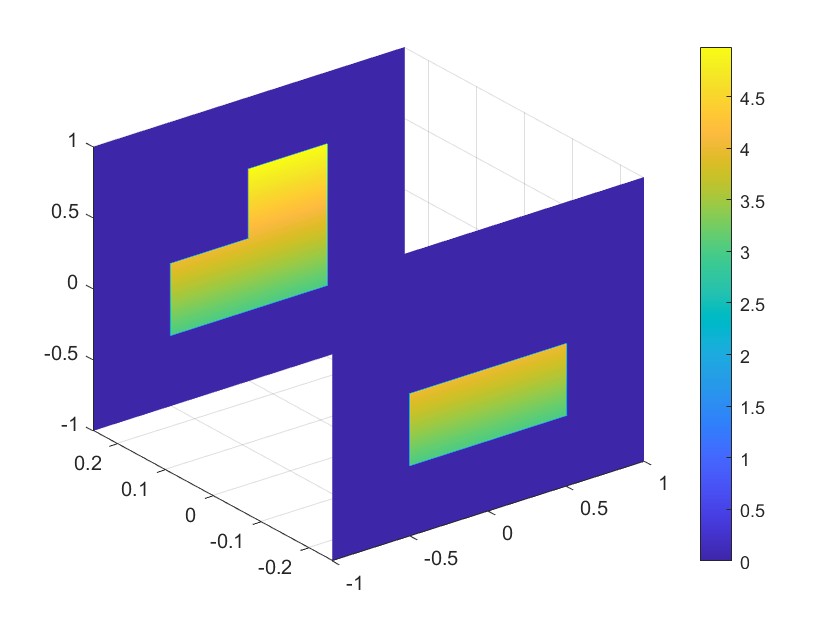}
        }&
        \subfigure[Slices of $J_3$ at $z_3=\pm 0.25$.]{
            \label{sourceJ3}
            \includegraphics[width=0.3\textwidth]{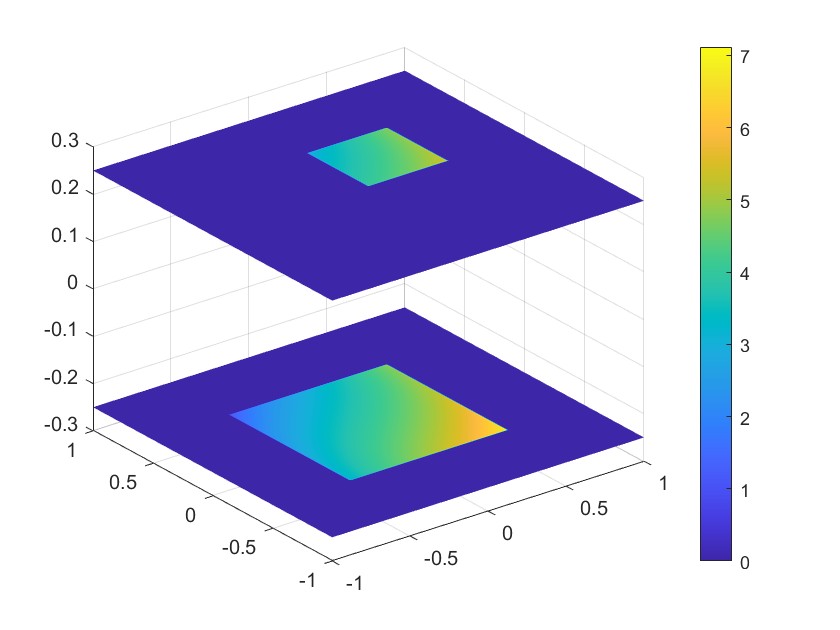}
        }
    \end{tabular}
    \caption{Slices of the true external electromagnetic source function $J=(J_1,J_2,J_3)^T$ in Example three.}
    \label{sourceJ}
\end{figure}

\begin{figure}[htbp]
   \centering
    \begin{tabular}{ccc}
        \subfigure[Slices at $z_1=\pm  0.25$, $L=51$.]{
            \label{IndE1-L_51-Omega_40}
        \includegraphics[width=0.3\textwidth]{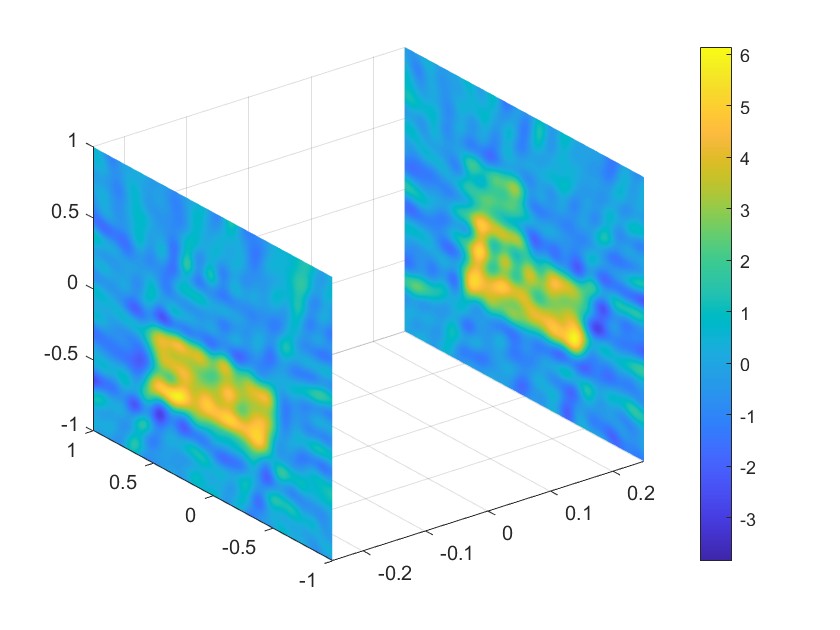}
        }&
        \subfigure[Slices at $z_1=\pm  0.25$, $L=151$.]{
            \label{IndE1-L_151-Omega_40-L}
            \includegraphics[width=0.3\textwidth]{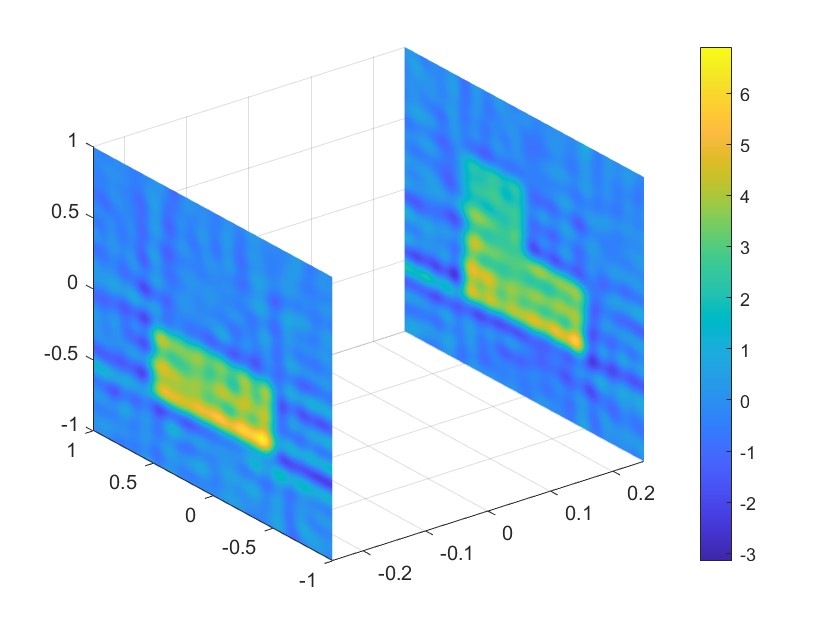}
        }&
        \subfigure[Slices at $z_1=\pm  0.25$, $L=251$.]{
            \label{IndE1-L_251-Omega_40}
            \includegraphics[width=0.3\textwidth]{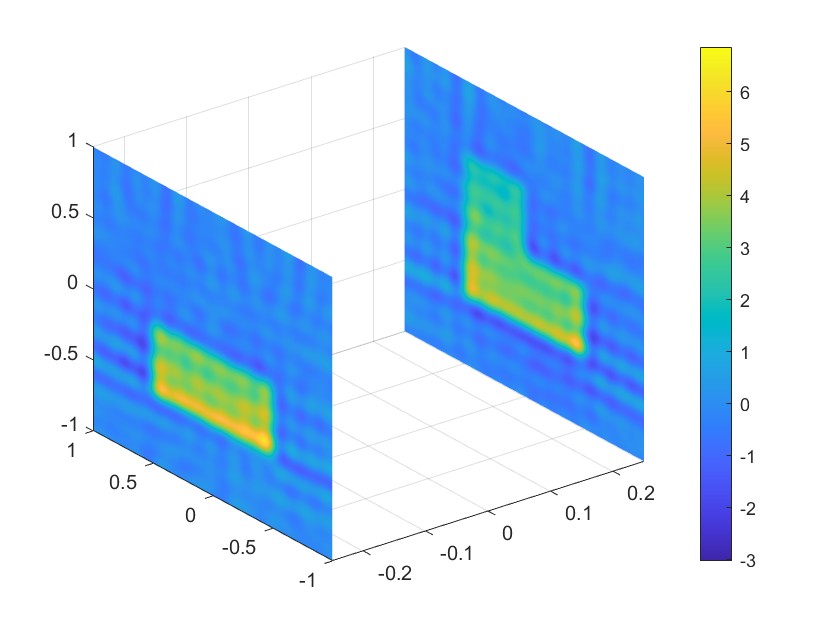}
        }\\
        \subfigure[Slices at $z_2=\pm  0.25$, $L=51$.]{
            \label{IndE2-L_51-Omega_40}
        \includegraphics[width=0.3\textwidth]{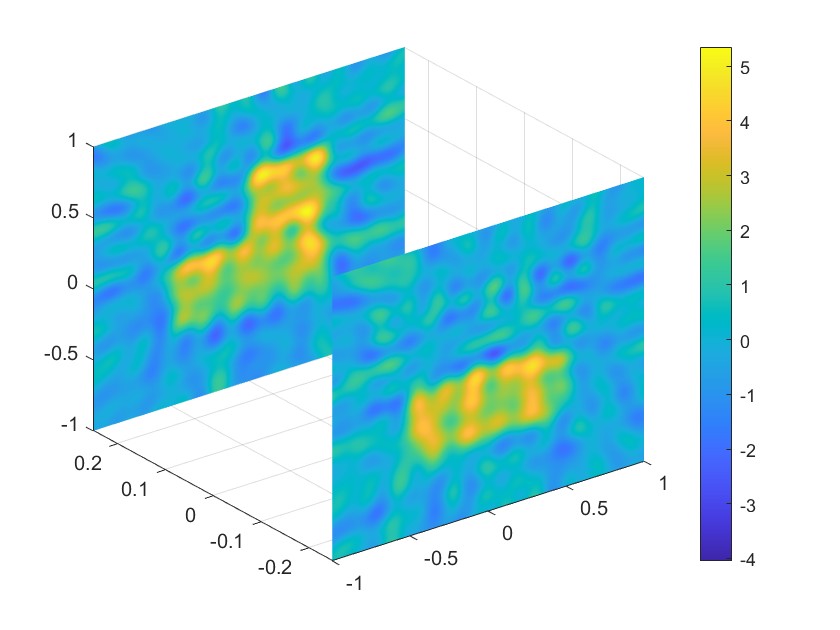}
        }&
        \subfigure[Slices at $z_2=\pm  0.25$, $L=151$.]{
            \label{IndE2-L_151-Omega_40-L}
            \includegraphics[width=0.3\textwidth]{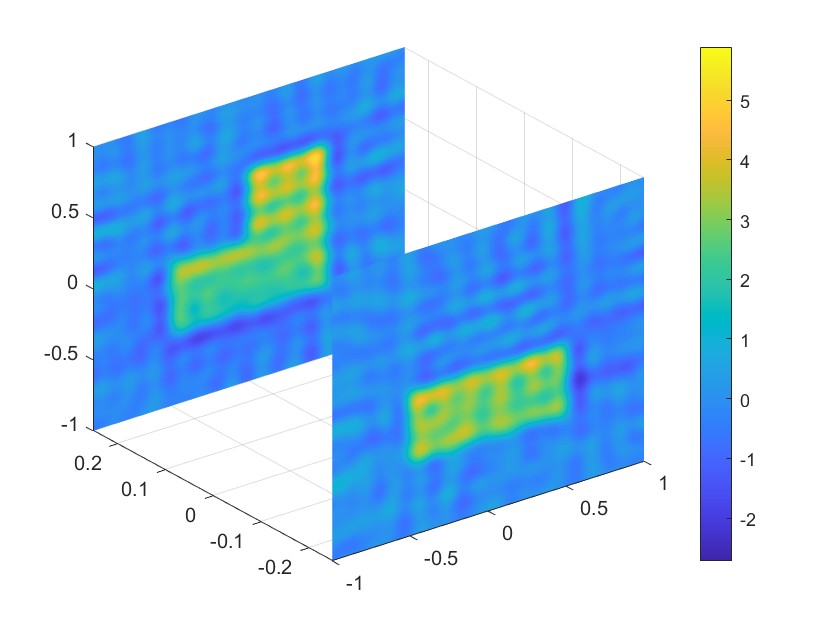}
        }&
        \subfigure[Slices at $z_2=\pm  0.25$, $L=251$.]{
            \label{IndE2-L_251-Omega_40}
            \includegraphics[width=0.3\textwidth]{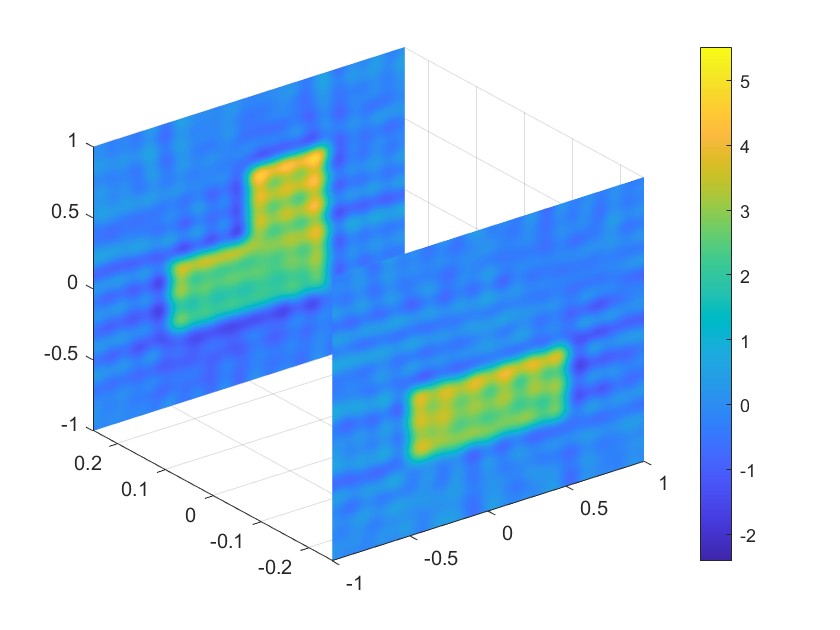}
        }\\       
        \subfigure[Slices at $z_3=\pm  0.25$, $L=51$.]{
            \label{IndE3-L_51-Omega_40}
        \includegraphics[width=0.3\textwidth]{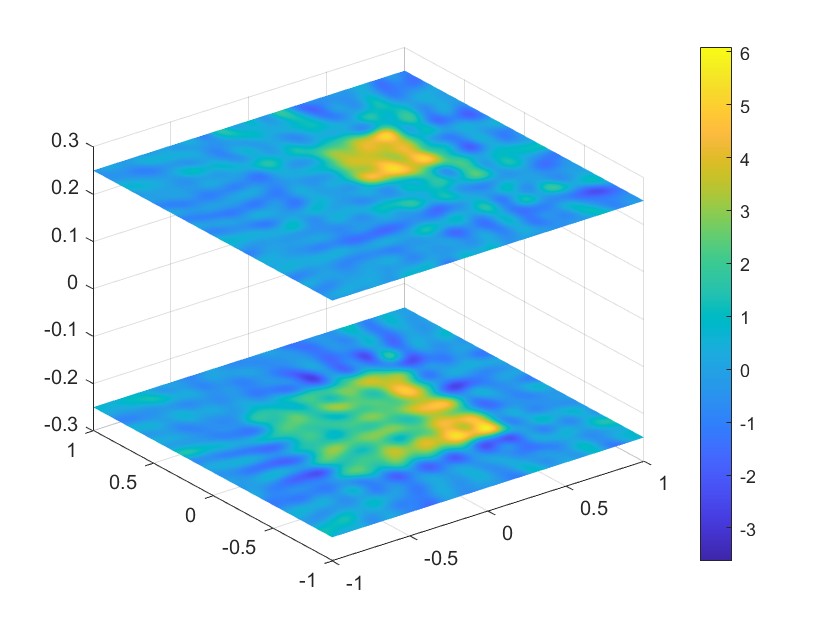}
        }&
        \subfigure[Slices at $z_3=\pm  0.25$, $L=151$.]{
            \label{IndE3-L_151-Omega_40-L}
            \includegraphics[width=0.3\textwidth]{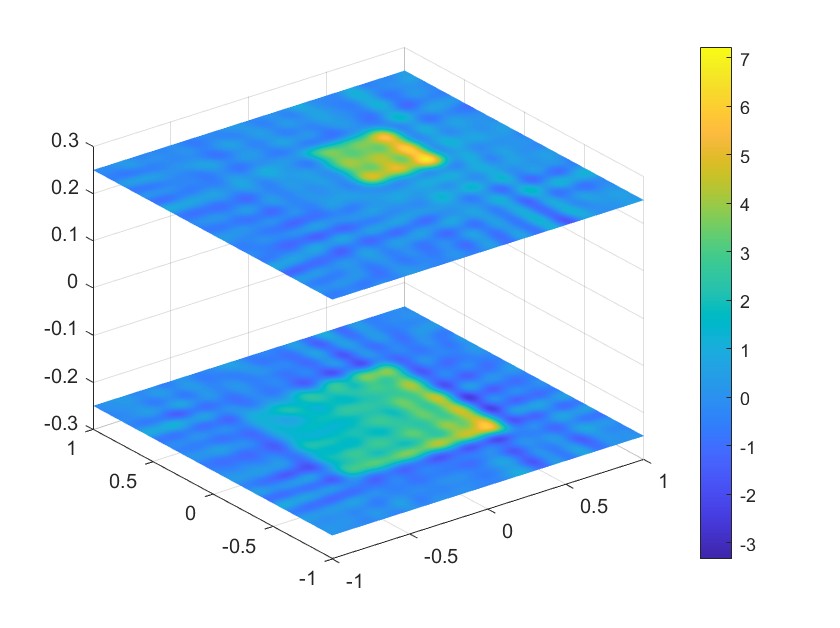}
        }&
        \subfigure[Slices at $z_3=\pm  0.25$, $L=251$.]{
            \label{IndE3-L_251-Omega_40}
            \includegraphics[width=0.3\textwidth]{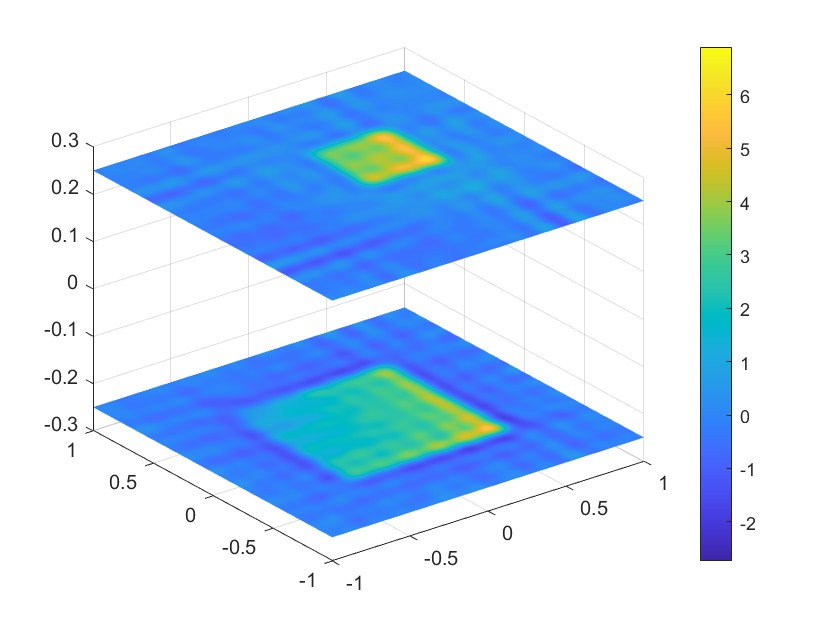}
        }
    \end{tabular}
    \caption{Reconstructions of $J=(J_1, J_2, J_3)^T$ by the indicator function $I_E$ with $k_{\Lambda}=40$ and different $L$. Top row: Reconstructions of $J_1$ by $(I_E)_{1}$. Middle row: Reconstructions of $J_2$ by $(I_E)_{2}$. Bottom row: Reconstructions of $J_3$ by $(I_E)_{3}$.}
    \label{IndE-Omega_40}
\end{figure}

\begin{figure}[htbp]
   \centering
    \begin{tabular}{ccc}
        \subfigure[Slices at $z_1=\pm  0.25$, $k_{\Lambda}=30$.]{
            \label{IndE1-L_151-Omega_30}
        \includegraphics[width=0.3\textwidth]{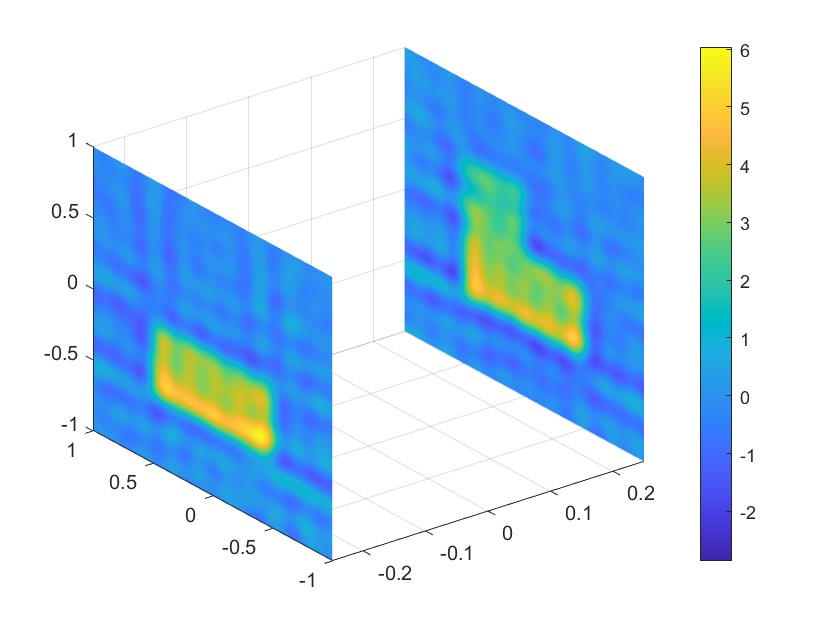}
        }&
        \subfigure[Slices at $z_1=\pm  0.25$, $k_{\Lambda}=40$.]{
            \label{IndE1-L_151-Omega_40-oemga}
            \includegraphics[width=0.3\textwidth]{figure/IndE1-N_151-Omega_40.jpg}
        }&
        \subfigure[Slices at $z_1=\pm  0.25$, $k_{\Lambda}=50$.]{
            \label{IndE1-L_151-Omega_50}
            \includegraphics[width=0.3\textwidth]{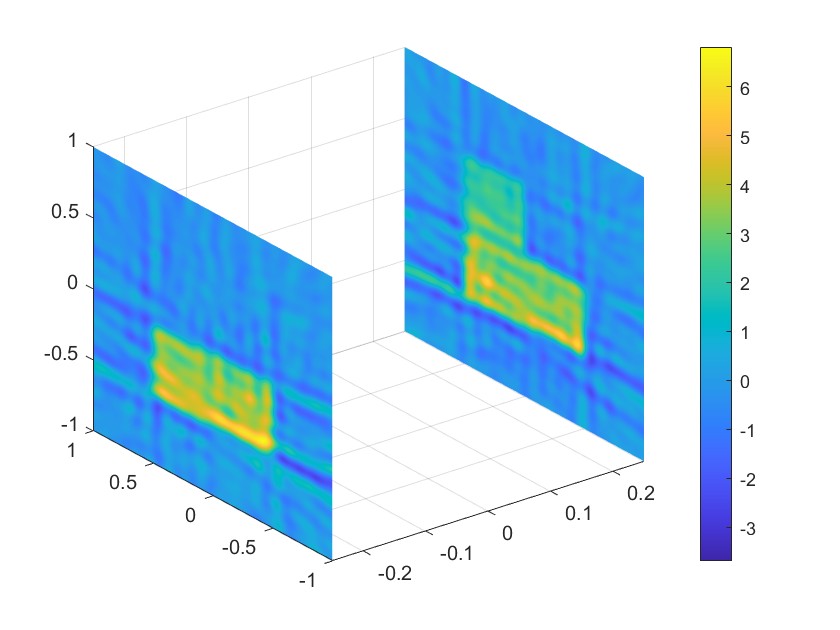}
        }\\
        \subfigure[Slices at $z_2=\pm  0.25$, $k_{\Lambda}=30$.]{
            \label{IndE2-L_151-Omega_30}
        \includegraphics[width=0.3\textwidth]{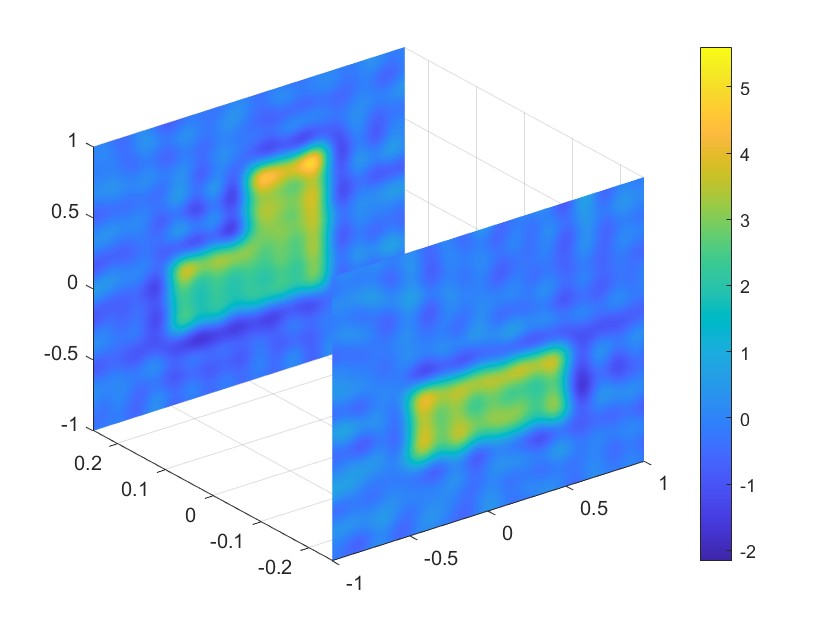}
        }&
        \subfigure[Slices at $z_2=\pm  0.25$, $k_{\Lambda}=40$.]{
            \label{IndE2-L_151-Omega_40-oemga}
            \includegraphics[width=0.3\textwidth]{figure/IndE2-N_151-Omega_40.jpg}
        }&
        \subfigure[Slices at $z_2=\pm  0.25$, $k_{\Lambda}=50$.]{
            \label{IndE2-L_151-Omega_50}
            \includegraphics[width=0.3\textwidth]{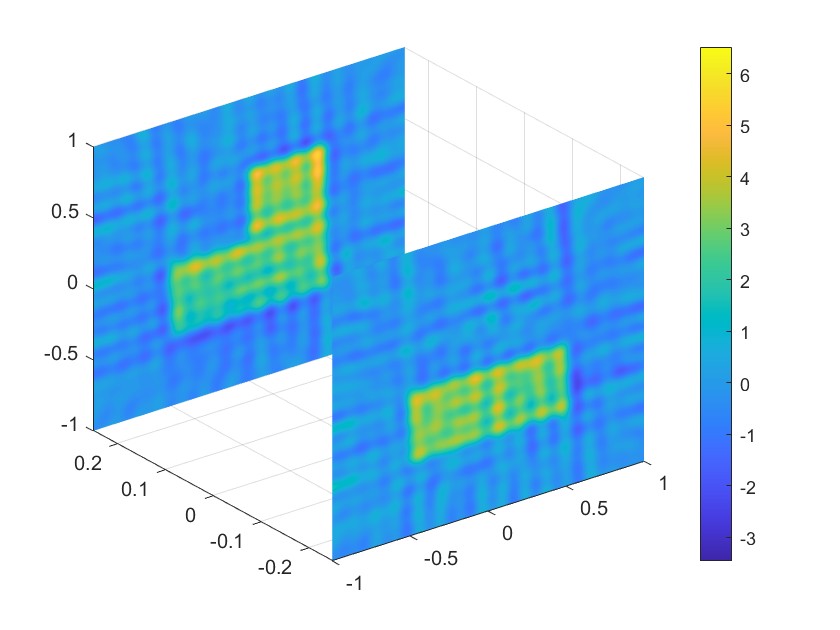}
        }\\       
        \subfigure[Slices at $z_3=\pm  0.25$, $k_{\Lambda}=30$.]{
            \label{IndE3-L_151-Omega_30}
        \includegraphics[width=0.3\textwidth]{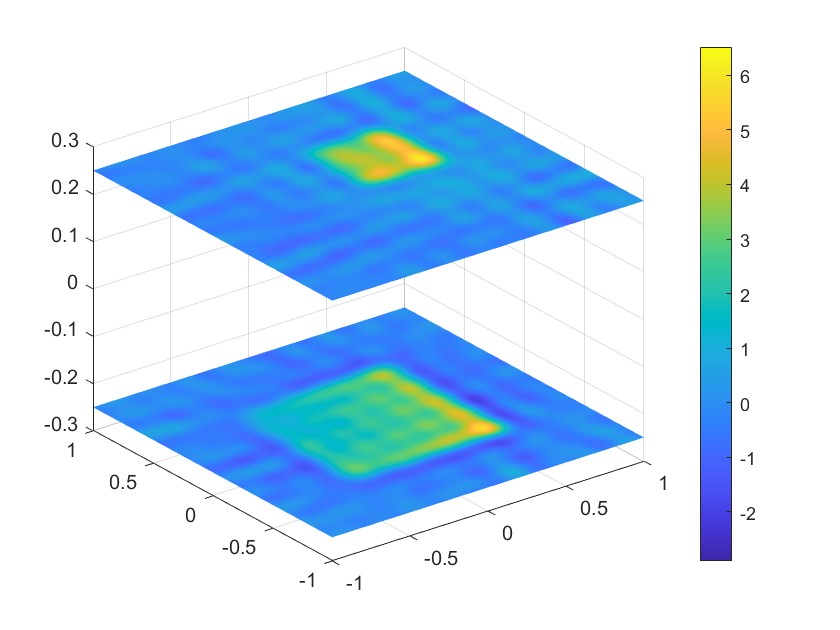}
        }&
        \subfigure[Slices at $z_3=\pm  0.25$, $k_{\Lambda}=40$.]{
            \label{IndE3-L_151-Omega_40-oemga}
            \includegraphics[width=0.3\textwidth]{figure/IndE3-N_151-Omega_40.jpg}
        }&
        \subfigure[Slices at $z_3=\pm  0.25$, $k_{\Lambda}=50$.]{
            \label{IndE3-L_151-Omega_50}
            \includegraphics[width=0.3\textwidth]{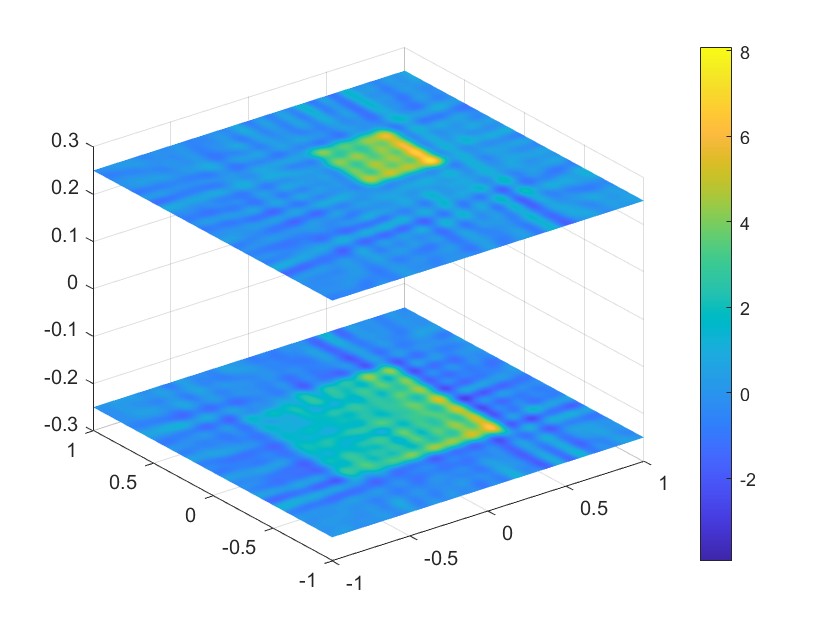}
        }
    \end{tabular}
    \caption{Reconstructions of $J=(J_1, J_2, J_3)^T$ by the indicator function $I_E$ with $L=151$ and different $k_\Lambda$. Top row: Reconstructions of $J_1$ by $(I_E)_{1}$. Middle row: Reconstructions of $J_2$ by $(I_E)_{2}$. Bottom row: Reconstructions of $J_3$ by $(I_E)_{3}$.}
    \label{IndE-L_151}
\end{figure}

\begin{itemize}
    \item \textbf{Example four: a source $J$ with nontrivial charge density $\rho\neq0$.}
\end{itemize}
In the forth example, we set $J(z)=(|z|^2-0.25)^2(1,1,1)^T,\ z\in\Omega$ with compact support $\Omega=\{z\in\mathbb R^3\big||z|\leq 0.5\}$. Straightforward calculations show that  
\ben
&{\rm div}J(z)=4(|z|^2-0.25)(z_1+z_2+z_3),\quad z\in\Om,\\
&{\rm curl}J(z)=4(|z|^2-0.25)(z_2-z_3, z_3-z_1, z_1-z_2)^T,\quad z\in\Om.
\enn
Obviously, ${\rm div}J=0$ in $\Om$ if and only if $|z|=0.5$ or $z_1+z_2+z_3=0$. 
We refer to the slices of true $J$ and ${\rm curl}J$, respectively, in the first and second rows of Figure \ref{IndH-L_151-Lambda_40}.
As shown in the third row of Figure \ref{IndH-L_151-Lambda_40}, we obtain a quite stable and high resolution reconstruction of ${\rm curl}\ J$ by the indicator $I_H$.
Moreover, by comparing the first and third rows of Figure \ref{IndH-L_151-Lambda_40} we observe that a rough reconstruction of the support $\Omega$ can be obtained by plotting $I_H$.
\begin{figure}[htbp]
   \centering
    \begin{tabular}{ccc}
     \subfigure[True $J_1$  at $z_1=0$.]{
            \label{J1inex4}
        \includegraphics[width=0.3\textwidth]{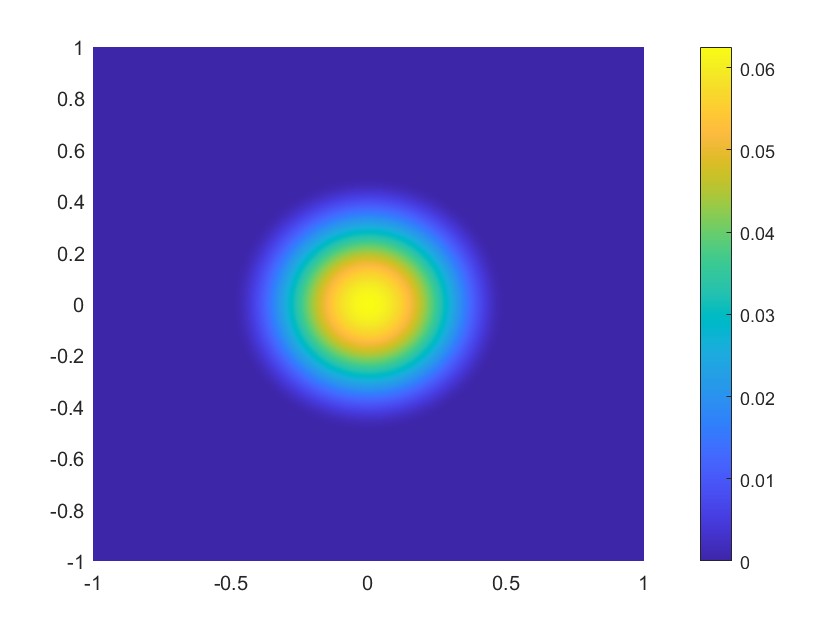}
        }&
        \subfigure[True $J_2$  at $z_2=0$.]{
            \label{J2inex4}
            \includegraphics[width=0.3\textwidth]{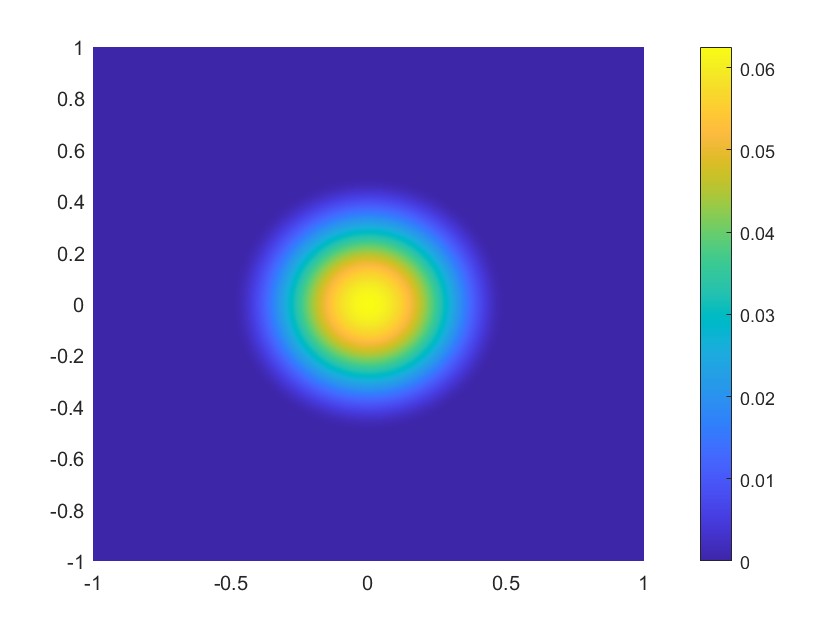}
        }&
        \subfigure[True $J_3$  at $z_3=0$.]{
            \label{J3inex4}
            \includegraphics[width=0.3\textwidth]{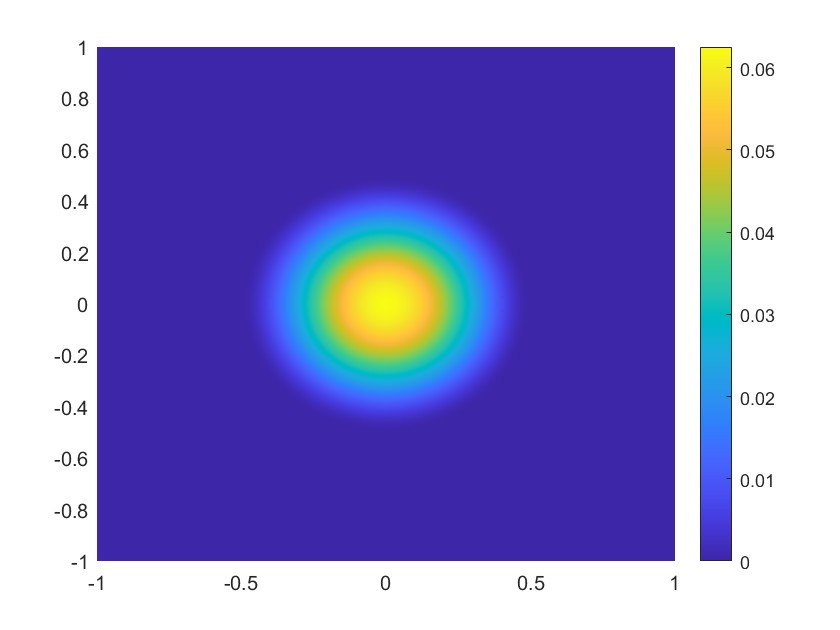}
        }\\
        \subfigure[True $({\rm curl}\ J)_1$  at $z_1=0$.]{
            \label{curlJ1}
        \includegraphics[width=0.3\textwidth]{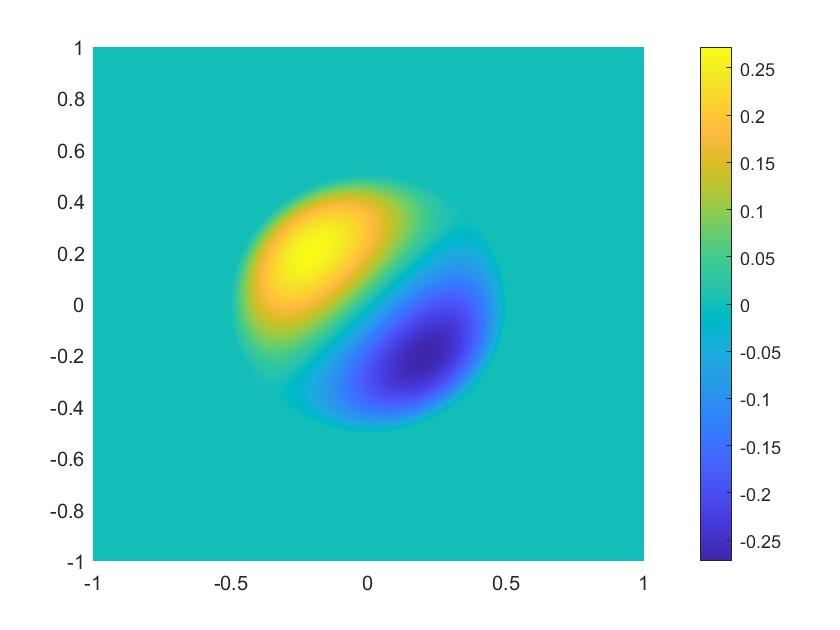}
        }&
        \subfigure[True $({\rm curl}\ J)_2$  at $z_2=0$.]{
            \label{curlJ2}
            \includegraphics[width=0.3\textwidth]{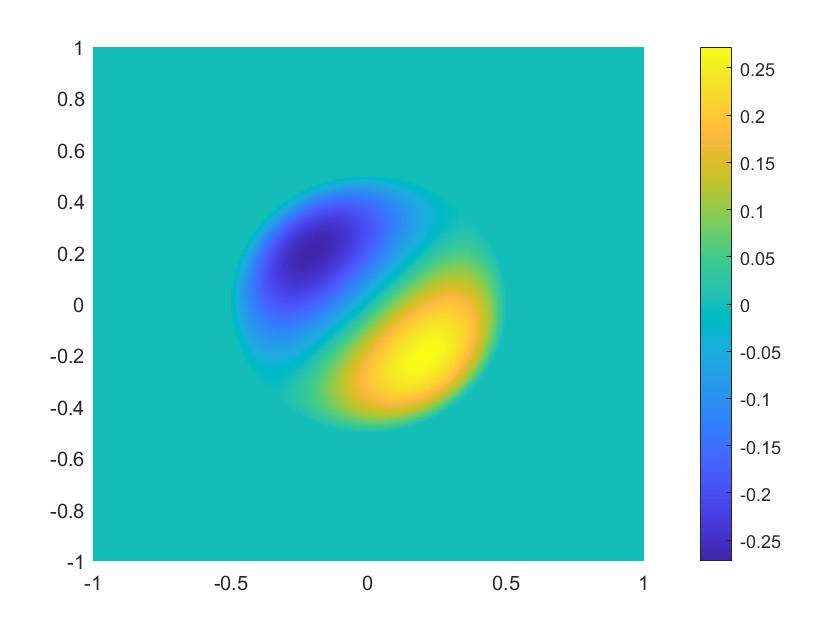}
        }&
        \subfigure[True $({\rm curl}\ J)_3$  at $z_3=0$.]{
            \label{curlJ3}
            \includegraphics[width=0.3\textwidth]{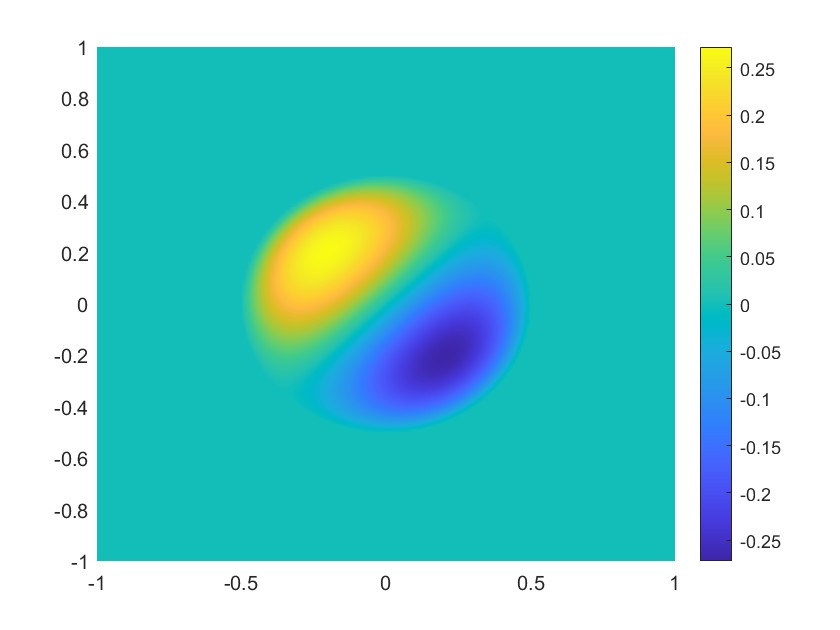}
        }\\
        \subfigure[Slice of $(I_H)_1$ at $z_1=0$.]{
            \label{IndH1-L_151-Omega_40}
        \includegraphics[width=0.3\textwidth]{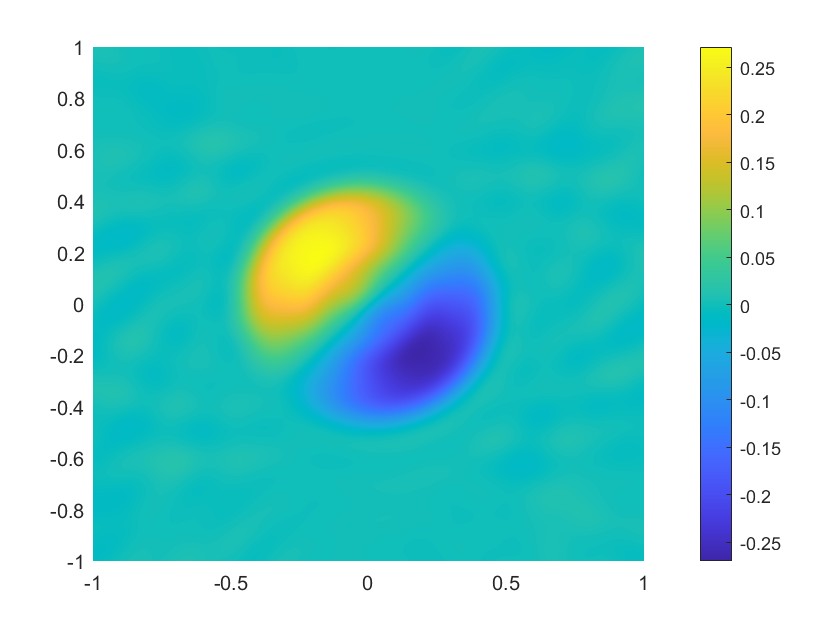}
        }&
        \subfigure[Slice of $(I_H)_2$ at $z_2=0$.]{
            \label{IndH2-L_151-Omega_40}
            \includegraphics[width=0.3\textwidth]{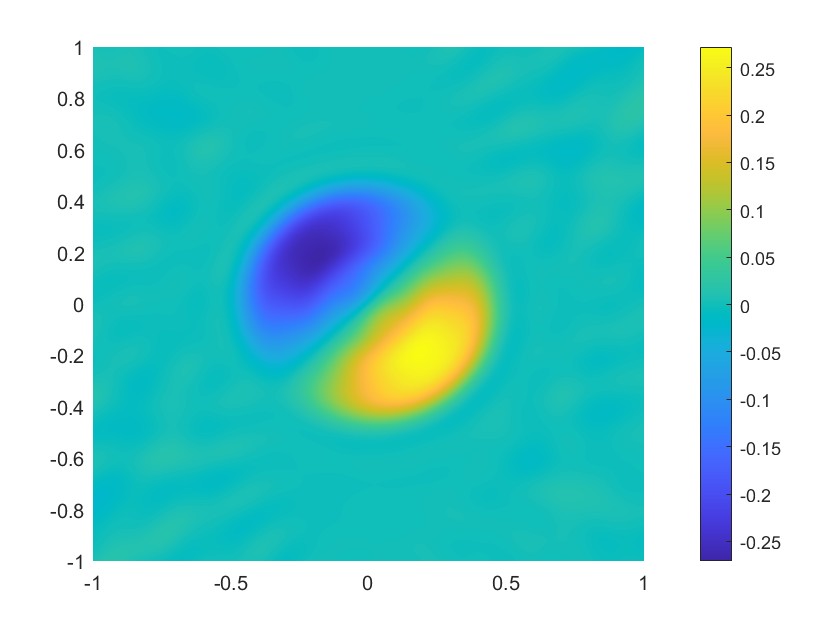}
        }&
        \subfigure[Slice of $(I_H)_3$ at $z_3=0$.]{
            \label{IndH3-L_151-Omega_40}
            \includegraphics[width=0.3\textwidth]{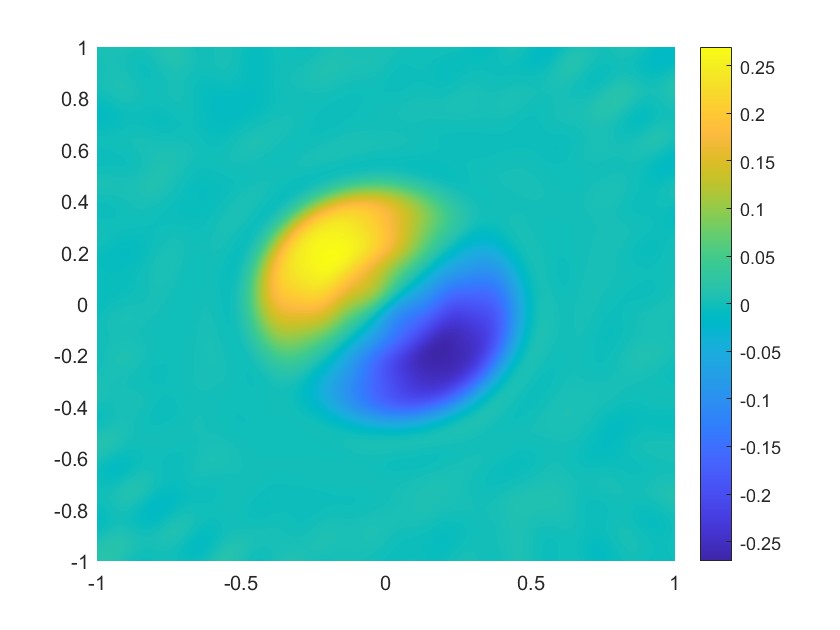}
        }
    \end{tabular}
    \caption{Reconstruction of ${\rm curl} J$ by $I_H$ with $L=151$ and $k_\Lambda=40$. }
    \label{IndH-L_151-Lambda_40}
\end{figure}


\section*{Acknowledgment}
The research of X. Liu is supported by the National Key R\&D Program of China under grant 2024YFA1012303 and the NNSF of China under grant 12371430.

\bibliographystyle{SIAM}

\end{document}